\documentclass{amsart}
\usepackage[utf8]{inputenc}
\usepackage{zeyu}
\newcommand{\ex}{\text{ex}}
\newcommand{\pl}{\mathcal{P}}
\title{Extremal planar graphs with no cycles of particular lengths}

\author{ERVIN GY\H{O}RI}
\address{Alfr\'ed R\'enyi Institute of Mathematics, Budapest, Hungary}
\email{gyori@renyi.hu}

\author{Xianzhi Wang}
\address{Middlebury College, Middlebury, VT 05753, United States}
\email{xianzhiw@middlebury.edu}

\author{Zeyu Zheng}
\address{School of Mathematical Sciences, Fudan University, Shanghai, China 200433}
\email{zeyuzheng19@fudan.edu.cn}

\begin{document}
\begin{abstract}
   In this paper we estimate the planar Tur\'an number $\mathrm{ex}_\mathcal{P}(n,H)$ of some graphs $H$,  i.e., the maximum number of edges in a planar graph $G$ of $n$ vertices not containing $H$ as a subgraph. We give a new, short proof when $H=C_5$, and study the cases when $G$ is bipartite or triangle-free and $H$ is a short even cycle. The proofs are mostly new applications or variants of the ``contribution method" introduced by Ghosh, Gy\H{o}ri, Martin, Paulos and Xiao  in \cite{ghosh2020planar}.
\end{abstract}
\maketitle

\section{Introduction}

All graphs considered in this paper are finite, undirected and simple.  We denote the vertex and the edge sets of a graph $G$ by $V(G)$ and $E(G)$, respectively, and the number of vertices and edges by $v_G$ and $e_G$, respectively,. The degree of a vertex $v$ is denoted by $d(v)$, and the minimum degree of a vertex in $G$ is denoted by $\delta(G)$. If $G$ is planar, we denote the set of its faces by $F(G)$, and we denote the number of its faces by $f_G$. Usually, we deal with $2$-connected graphs, and the length of $f$ is denoted by $l(f)$ as in \cite{west2020combinatorial}.

The classical problem in extremal graph theory is to determine the value of Tur\'an number $\ex(n,H)$, which is the maximum number of edges in a graph of $n$ vertices not containing $H$ as a subgraph. The study of planar Tur\'an number $\ex_\pl (n,H)$, initiated by Dowden in 2016 \cite{dowden2016extremal}, is a variant of Tur\'an-type problem. The planar Tur\'an number $\ex_\pl(n,H)$ is the maximum number of edges in a planar graph of $n$ vertices without $H$ as a subgraph. He determined the planar Tur\'an number of $C_4$ and $C_5$ in that paper. Later in 2019, Lan and Shi determined the planar Tur\'an number of $\Theta_4$ and $\Theta_5$ in \cite{lan2019extremal}.

In 2020, Ghosh, Gy\H{o}ri, Martin, Paulos and Xiao introduced the triangular block contribution method in \cite{ghosh2020planar}, a ``slight extension" of the discharging method, which is a useful tool to determine the value of $\ex_\pl (n, C_6)$. Using this technique, we first provide a short proof of \begin{restatable}[Dowden\cite{dowden2016extremal}]{theorem}{thmCfive}
\label{thm:C5}
$$\mathrm{ex}_\mathcal{P}(n,C_5)\le\frac{12n-33}{5}$$
for all $n\ge11$.
\end{restatable} We've also found a new extremal construction showing sharpness of Theorem \ref{thm:C5} by means of our proof technique.

Then, we introduce a similar method regarding quadrangular blocks. Using this method, we first study the maximum number of edges in bipartite planar graphs without $C_6$. Although $K_{2,n-2}$ is a trivially maximum construction since it reaches the maximum number of edges, and it does not contain any cycle other than $C_4$. However, it is a very special graph that has a lot of degree $2$ vertices. Therefore, we include the number of degree $2$ vertices in the upper bound of edges. Also, we would like to avoid triviality from $K_{2,n-2}$-like configurations by deleting degree $2$ vertices whose two neighbors both have degree greater than $3$ as long as each deletion does not create new degree $2$ vertices.
\begin{restatable}{theorem}{thmBiCSix}
\label{thm:biC6}
Let $G$ be a $C_6$-free planar bipartite graph on $n$ vertices. Then $\delta(G)\le 2$, and if any degree $2$ vertex $v$ in $G$ has a neighbor of degree at most $3$, then
$$
    e_G \leq \frac{3}{2}n+\frac{1}{2}k+\frac{1}{4}e_{2,3}-4,
$$
for all $n\ge6$, where $k$ is the number of degree $2$ vertices in $G$ and $e_{2,3}$ is the number of edges $xy$ in $G$ such that $d(x)=2$ and $d(y)=3$.
\end{restatable} Based on our proof technique, We've found constructions for infinitely many graphs showing sharpness of Theorem \ref{thm:biC6}.

This method also allows us to prove some other results:

\begin{restatable}{theorem}{thmBiCEight}
\label{thm:biC8}
Let $G$ be a $C_8$-free planar bipartite graph with $\delta(G) \geq 3$ on $n$ vertices. Then 
$$
    e_G \leq \frac{5}{3}n-\frac{10}{3}.
$$ The equality holds for infinitely many integers $n$.
\end{restatable}
\begin{restatable}{theorem}{thmBiCEightCTen}
\label{thm:C8C10}
Let $G$ be a planar bipartite graph on $n$ vertices which does not contain $C_8$ or $C_{10}$ and let $\delta(G)\ge 3$. Then $$
    e_G \leq \frac{18}{11}n-\frac{84}{11}.
$$ The equality holds for infinitely many integers $n$.
\end{restatable}
\begin{restatable}{theorem}{thmTriCSix}
\label{thm:triC6}
Let $G$ be a $C_6$-free triangle-free planar graph with $\delta(G) \geq 3$ on $n$ vertices. Then 
$$
    e_G \leq \left\lfloor \frac{9}{5}n - 4 \right\rfloor.
$$ The equality holds for infinitely many integers $n$. 
\end{restatable}
\begin{restatable}{theorem}{thmTriCEight}
\label{thm:triC8}
Let $G$ be a $C_8$-free triangle-free planar graph with $\delta(G) \geq 3$ on $n$ vertices. Then 
$$
    e_G \leq \frac{81}{44}n -\frac{105}{22}.
$$
\end{restatable}

\section{A new proof of \texorpdfstring{\cref{thm:C5}}{Theorem \ref{thm:C5}}}
\label{sec:C5}
\subsection{Definitions and preparatory propositions}
\begin{definition}
\label{def:triblock}
Let $G$ be a plane graph. For an edge $e\in E(G)$, if it's not contained in any of the $3$-faces of $G$, then we call it a {\bf trivial triangular block}. Otherwise, we perform the following algorithm to construct a {\bf triangular block} $B$.
\begin{algorithm}[h] 
$B\gets(V(e),e)$\;
    \While{there exists an edge in $B$ such that it is in a bounded $3$-face of $G$ which is not contained in $B$}
    {add all the edges of such bounded $3$-faces to $B$;}
    Output $B$;
\end{algorithm}

Notice that the resulting triangular block does not depend on the choice of starting edge as long as the starting edge is in the triangular block.
\end{definition}

Now, we show that there are only four possible kinds of blocks in a $C_5$-free plane graph. This needs the following proposition.

\begin{proposition}
A triangular block in a $C_5$-free plane graph contains at most $4$ vertices.
\end{proposition}
\begin{proof}
Consider a triangular block having more than $4$ vertices. By the definition of triangular blocks, it must be initiated from two adjacent $3$-faces sharing one common edge.
%(since two adjacent $3$-faces sharing two common edges would create degree two vertex and multi-edge).
The only way to attach another $3$-face to it while not creating a $C_5$ or multiple edges is to add an edge so that the triangular block becomes a $K_4$. Now adding a $3$-face in any other manner would create a $C_5$. %\textbf{\color{blue}{need to mention degree at least 3 before this proposition}}

%Assume for contradiction that there exists a triangular block $B'$ in a $C_5$ free plane graph with $n \geq 5$ number of vertices. We delete vertices on the unbounded face of $B'$ one by one until there are exactly $5$ vertices left in $B'$, and we denote the new triangular block by $B$. Without loss of generality, since when we construct the blocks, only $3$-faces are considered, there are only two possibilities for $B$. One is a $K_4$ with a triangle attached, and the other is (figures needed). Then in both cases, we could find a $C_5$, contradiction.
\end{proof}

Now we describe all the possible triangular blocks. As shown in Figure \ref{fig:triBlocks}, other than the trivial triangular block $K_2$, there is one triangular block on $3$ vertices $K_3$ and two triangular blocks on $4$ vertices $K_4$ and $\Theta_4$. 
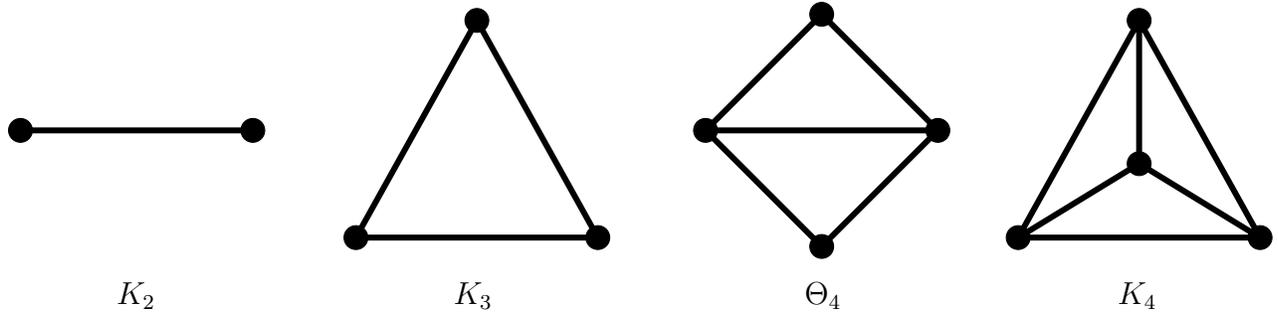
\begin{figure}[H]
\centering
\begin{tikzpicture}[x=0.75pt,y=0.75pt,yscale=-1,xscale=1]
%uncomment if require: \path (0,300); %set diagram left start at 0, and has height of 300

%Shape: Triangle [id:dp9659950400291835] 
\draw  [line width=2.25]  (592,88) -- (653,197.61) -- (531,197.61) -- cycle ;
%Straight Lines [id:da49186325967081346] 
\draw [line width=2.25]    (592,88) -- (592,160.3) ;
%Straight Lines [id:da9487540189981805] 
\draw [line width=2.25]    (592,160.3) -- (531,197.61) ;
%Straight Lines [id:da32291680213956453] 
\draw [line width=2.25]    (653,197.61) -- (592,160.3) ;
%Shape: Square [id:dp7708663438470886] 
\draw  [line width=2.25]  (431.91,84.91) -- (490.5,143.5) -- (431.91,202.09) -- (373.32,143.5) -- cycle ;
%Straight Lines [id:da862841283103198] 
\draw [line width=2.25]    (373.32,143.5) -- (490.5,143.5) ;
%Shape: Circle [id:dp5844354377788905] 
\draw  [fill={rgb, 255:red, 0; green, 0; blue, 0 }  ,fill opacity=1 ] (586,88) .. controls (586,84.69) and (588.69,82) .. (592,82) .. controls (595.31,82) and (598,84.69) .. (598,88) .. controls (598,91.31) and (595.31,94) .. (592,94) .. controls (588.69,94) and (586,91.31) .. (586,88) -- cycle ;
%Shape: Circle [id:dp47135170799204595] 
\draw  [fill={rgb, 255:red, 0; green, 0; blue, 0 }  ,fill opacity=1 ] (525,197.61) .. controls (525,194.3) and (527.69,191.61) .. (531,191.61) .. controls (534.31,191.61) and (537,194.3) .. (537,197.61) .. controls (537,200.93) and (534.31,203.61) .. (531,203.61) .. controls (527.69,203.61) and (525,200.93) .. (525,197.61) -- cycle ;
%Shape: Circle [id:dp7921328717549245] 
\draw  [fill={rgb, 255:red, 0; green, 0; blue, 0 }  ,fill opacity=1 ] (647,197.61) .. controls (647,194.3) and (649.69,191.61) .. (653,191.61) .. controls (656.31,191.61) and (659,194.3) .. (659,197.61) .. controls (659,200.93) and (656.31,203.61) .. (653,203.61) .. controls (649.69,203.61) and (647,200.93) .. (647,197.61) -- cycle ;
%Shape: Circle [id:dp9590064951385151] 
\draw  [fill={rgb, 255:red, 0; green, 0; blue, 0 }  ,fill opacity=1 ] (367.32,143.5) .. controls (367.32,140.19) and (370.01,137.5) .. (373.32,137.5) .. controls (376.64,137.5) and (379.32,140.19) .. (379.32,143.5) .. controls (379.32,146.81) and (376.64,149.5) .. (373.32,149.5) .. controls (370.01,149.5) and (367.32,146.81) .. (367.32,143.5) -- cycle ;
%Shape: Circle [id:dp007927085963679525] 
\draw  [fill={rgb, 255:red, 0; green, 0; blue, 0 }  ,fill opacity=1 ] (484.5,143.5) .. controls (484.5,140.19) and (487.19,137.5) .. (490.5,137.5) .. controls (493.81,137.5) and (496.5,140.19) .. (496.5,143.5) .. controls (496.5,146.81) and (493.81,149.5) .. (490.5,149.5) .. controls (487.19,149.5) and (484.5,146.81) .. (484.5,143.5) -- cycle ;
%Shape: Circle [id:dp3640195570138509] 
\draw  [fill={rgb, 255:red, 0; green, 0; blue, 0 }  ,fill opacity=1 ] (425.91,202.09) .. controls (425.91,198.77) and (428.6,196.09) .. (431.91,196.09) .. controls (435.23,196.09) and (437.91,198.77) .. (437.91,202.09) .. controls (437.91,205.4) and (435.23,208.09) .. (431.91,208.09) .. controls (428.6,208.09) and (425.91,205.4) .. (425.91,202.09) -- cycle ;
%Shape: Circle [id:dp1075956281395325] 
\draw  [fill={rgb, 255:red, 0; green, 0; blue, 0 }  ,fill opacity=1 ] (425.91,84.91) .. controls (425.91,81.6) and (428.6,78.91) .. (431.91,78.91) .. controls (435.23,78.91) and (437.91,81.6) .. (437.91,84.91) .. controls (437.91,88.23) and (435.23,90.91) .. (431.91,90.91) .. controls (428.6,90.91) and (425.91,88.23) .. (425.91,84.91) -- cycle ;
%Shape: Circle [id:dp610029723109272] 
\draw  [fill={rgb, 255:red, 0; green, 0; blue, 0 }  ,fill opacity=1 ] (586,160.3) .. controls (586,156.98) and (588.69,154.3) .. (592,154.3) .. controls (595.31,154.3) and (598,156.98) .. (598,160.3) .. controls (598,163.61) and (595.31,166.3) .. (592,166.3) .. controls (588.69,166.3) and (586,163.61) .. (586,160.3) -- cycle ;
%Straight Lines [id:da7798225289157952] 
\draw [line width=2.25]    (27.82,143.5) -- (145,143.5) ;
%Shape: Circle [id:dp8157271254987504] 
\draw  [fill={rgb, 255:red, 0; green, 0; blue, 0 }  ,fill opacity=1 ] (21.82,143.5) .. controls (21.82,140.19) and (24.51,137.5) .. (27.82,137.5) .. controls (31.14,137.5) and (33.82,140.19) .. (33.82,143.5) .. controls (33.82,146.81) and (31.14,149.5) .. (27.82,149.5) .. controls (24.51,149.5) and (21.82,146.81) .. (21.82,143.5) -- cycle ;
%Shape: Circle [id:dp9226814486426016] 
\draw  [fill={rgb, 255:red, 0; green, 0; blue, 0 }  ,fill opacity=1 ] (139,143.5) .. controls (139,140.19) and (141.69,137.5) .. (145,137.5) .. controls (148.31,137.5) and (151,140.19) .. (151,143.5) .. controls (151,146.81) and (148.31,149.5) .. (145,149.5) .. controls (141.69,149.5) and (139,146.81) .. (139,143.5) -- cycle ;
%Shape: Triangle [id:dp5949960283309748] 
\draw  [line width=2.25]  (258,88) -- (319,197.61) -- (197,197.61) -- cycle ;
%Shape: Circle [id:dp554060485897476] 
\draw  [fill={rgb, 255:red, 0; green, 0; blue, 0 }  ,fill opacity=1 ] (252,88) .. controls (252,84.69) and (254.69,82) .. (258,82) .. controls (261.31,82) and (264,84.69) .. (264,88) .. controls (264,91.31) and (261.31,94) .. (258,94) .. controls (254.69,94) and (252,91.31) .. (252,88) -- cycle ;
%Shape: Circle [id:dp4814117470423771] 
\draw  [fill={rgb, 255:red, 0; green, 0; blue, 0 }  ,fill opacity=1 ] (191,197.61) .. controls (191,194.3) and (193.69,191.61) .. (197,191.61) .. controls (200.31,191.61) and (203,194.3) .. (203,197.61) .. controls (203,200.93) and (200.31,203.61) .. (197,203.61) .. controls (193.69,203.61) and (191,200.93) .. (191,197.61) -- cycle ;
%Shape: Circle [id:dp6547116420414438] 
\draw  [fill={rgb, 255:red, 0; green, 0; blue, 0 }  ,fill opacity=1 ] (313,197.61) .. controls (313,194.3) and (315.69,191.61) .. (319,191.61) .. controls (322.31,191.61) and (325,194.3) .. (325,197.61) .. controls (325,200.93) and (322.31,203.61) .. (319,203.61) .. controls (315.69,203.61) and (313,200.93) .. (313,197.61) -- cycle ;
% Text Node
\draw (75,219) node [anchor=north west][inner sep=0.75pt]  [font=\Large] [align=left] {$K_2$};
% Text Node
\draw (245,219) node [anchor=north west][inner sep=0.75pt]  [font=\Large] [align=left] {$K_3$};
% Text Node
\draw (422,219) node [anchor=north west][inner sep=0.75pt]  [font=\Large] [align=left] {$\Theta_4$};
% Text Node
\draw (580,219) node [anchor=north west][inner sep=0.75pt]  [font=\Large] [align=left] {$K_4$};
\end{tikzpicture}
\caption{All the possible triangular blocks}
\label{fig:triBlocks}
\end{figure}

\begin{definition}
\label{def:triSettings}
Let $G$ be a plane graph and $B$ be a triangular block.
\begin{enumerate}
    \item A vertex in $G$ is called a {\bf triangular block junction vertex} if it's contained in at least two different triangular blocks in $G$.
    \item A vertex or an edge of $B$ is called {\bf triangular block exterior} if it lies on the boundary of the exterior face of $B$ and it's called {\bf triangular block interior} otherwise. A face is called {\bf triangular block exterior} if it is not contained by $B$ while it has at least one common edge with $B$ and it's called {\bf triangular block interior} if it's contained by $B$.
\end{enumerate}
\end{definition}

In this section, for the sake of brevity in expression, {\bf triangular junction vertex} is abbreviated as {\bf junction vertex}, {\bf triangular block exterior} is abbreviated as {\bf exterior} and {\bf triangular block interior} is abbreviated as {\bf interior}.

\begin{definition}
\label{def:tricontributionvertex}
Let $G$ be a plane graph and $B$ is a triangular block in $G$. We denote the contribution to the vertex number of a vertex $v$ in $B$ by $n_B(v)$, and define it as$$
n_B(v) = \frac{1}{\#\textit{ triangular blocks in } G \textit{ containing } v}.
$$
The contribution of $B$ to the vertex number is defined as $$v(B)=\sum_{v\in V(B)}n_B(v).$$
\end{definition}

By definition we have $$
v_G=\sum_{v\in V(G)} 1=\sum_{v\in V(G)} \sum_{B\ni v}n_B(v)=\sum_{B\in\mathcal{T}}\sum_{v\in V(B)}n_B(v)=\sum_{B\in\mathcal{T}}v(B),
$$ where $\mathcal{T}$ is the family of all triangular blocks in $G$.

We define the contribution of $B$ to the edge number as the number of edges in $B$, and we denote it by $e(B)$. Since the triangular blocks are edge-disjoint by definition, and each edge in $G$ is included in a triangular block, we have $$e_G=\sum_{B\in\mathcal{T}}e(B).$$

Each face in $G$ is either an interior face of a unique triangular block or an exterior face of some triangular blocks. We denote the interior faces in $G$ by $F_I(G)$ and we denote the exterior faces in $G$ by $F_E(G)$. For an exterior face $f$, if its boundary contains exactly two consecutive exterior edges of a triangular block $B$ that is a $K_4$, we replace them by the other exterior edge of $B$ to get a smaller cycle. We do this as long as the cycle still contains two consecutive exterior edges from the same $K_4$ triangular block. We call the resulting cycle \textbf{exterior pseudoface} and denote it by $C_f$ and its length by $l'(f)$. A similar concept can be found in \cite{ghosh2020planar2} as face \textbf{refinement}.

\begin{remark}
Observe that if there is no triangular block being $K_4$ satisfying the previous description, an exterior pseudoface is just a face in the ordinary sense.
\end{remark}

\begin{definition}
\label{def:triContributionFace}
Let $G$ be a plane graph and $B$ is a triangular block in $G$. For each exterior edge $e$ of $B$, we denote its contribution to the face number in $B$ by $f_B(e)$. Note that $e$ can be the edge of at most two exterior faces, we define it as follows.
\begin{enumerate}
    \item If $e$ is the only edge of a $K_2$ triangular block and it's contained in the boundary of two exterior faces $f_1$ and $f_2$, then let $f_B(e)=1/l'(f_1)+1/l'(f_2)$.
    \item If two consecutive exterior edges $e_1,e_2$ of a $K_4$ triangular block $B$ are contained in the boundary of the same exterior face $f$, then let $f_B(e_1)+f_B(e_2)=1/l'(f)$.
    \item Otherwise, let $f_B(e)=1/l'(f)$ where $f$ is the exterior face containing $e$.
\end{enumerate}
The contribution of $B$ to the face number is defined as $$f(B)=\#\textit{ interior faces in $B$ } + \sum_{\textit{ \shortstack{e is an\\ exterior edge\\ of $B$ }}}f_B(e).$$
\end{definition}

 By definition we have \begin{align*}
    f_G& =\sum_{f\in F(G)}1=\sum_{f\in F_I(G)}1+\sum_{f\in F_E(G)}1
=\sum_{B\in\mathcal{T}}\#\textit{ interior faces in $B$}
+\sum_{f\in F_E(G)}\frac{l'(f)}{l'(f)}\\
 & =\sum_{B\in\mathcal{T}}\#\textit{ interior faces in $B$}
+\sum_{B\in\mathcal{T}}\sum_{\textit{\shortstack{ e is an \\ exterior edge\\ of $B$}}}f_B(e)
=\sum_{B\in\mathcal{T}}f(B)
 \end{align*}

\subsection{Proof of Theorem \ref{thm:C5}}~

We may assume that $\delta(G)\ge 3$ and $G$ is $2$-connected. For other graphs not satisfying this assumption, technical details can be found in \cite{dowden2016extremal}. Now our goal is the following theorem:
\begin{theorem}
\label{thm:C5Essential}
Let $G$ be a $2$-connected $C_5$-free planar graph with $\delta(G)\ge 3$ on $n$ vertices. Then $$
e_G\le \frac{12n-33}{5}.
$$ 
\end{theorem}

The extremal construction below shows that the bound in Theorem \ref{thm:C5Essential} is sharp. In the extremal construction, there are $15t^2-6$ vertices and $36t^2-21$ edges, which satisfy that $e_G=(12v_G-33)/5$.

\begin{figure}[H]
    \centering
    \includegraphics{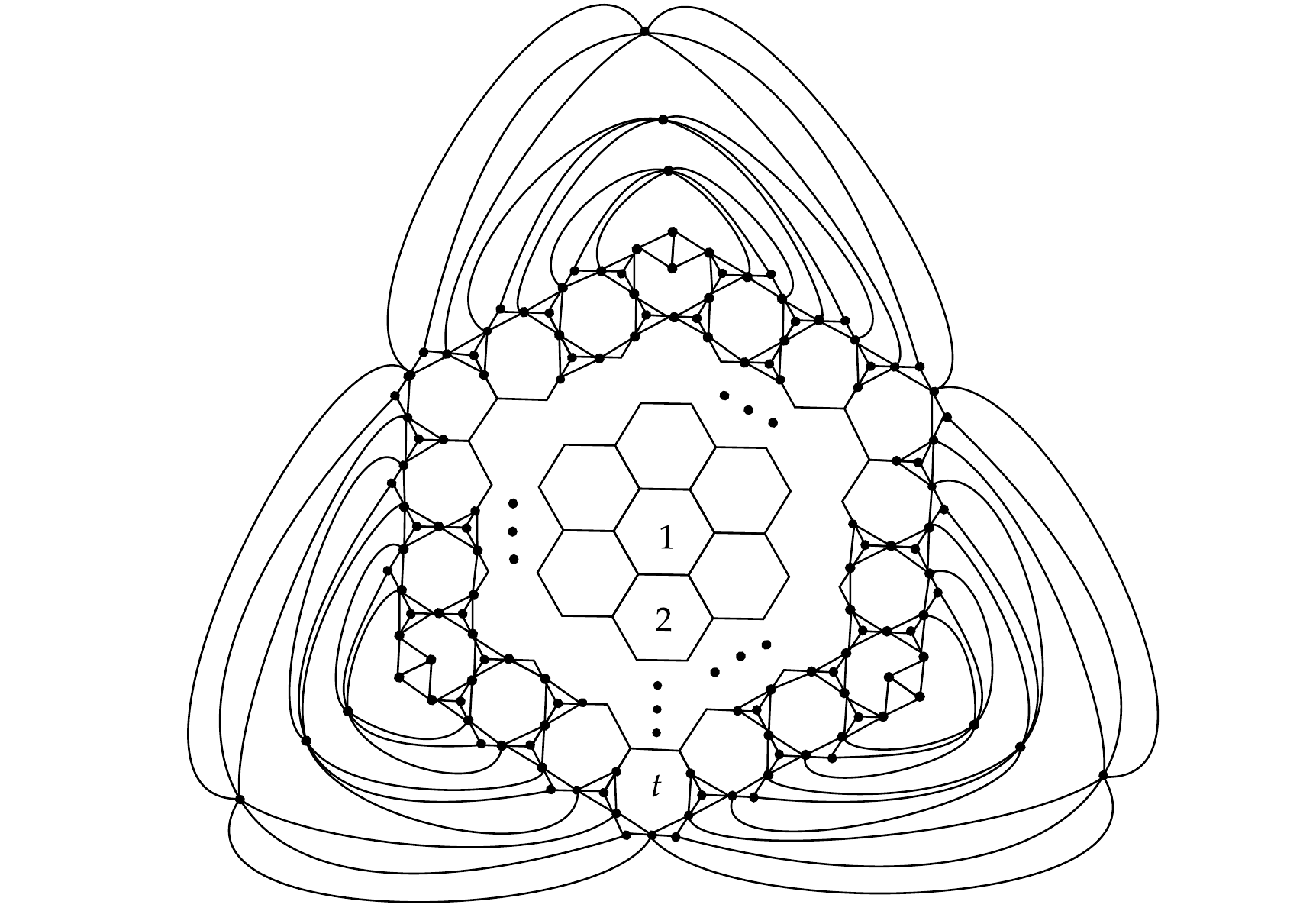}
    \caption{New Extremal Construction of Theorem \ref{thm:C5}.}
    \label{fig:C5Extremal}
\end{figure}

\begin{remark}
This construction comes from the proof of Lemma \ref{lem:C5}. In the extremal graph, each triangular block has $0$ ``contribution".
\end{remark}

The following proposition is useful in the proof.

\begin{proposition}
\label{pro:C5}
Let $G$ be a $2$-connected $C_5$-free plane graph with $\delta(G)\ge 3$ and $B$ is a nontrivial triangular block, then its exterior pseudofaces have length at least $6$.
\end{proposition}

\begin{proof}
Denote this pseudoface by $C_f$. $C_f$ cannot have length $3$. If $C_f$ is an exterior face in the ordinary sense, then by definition of triangular blocks, $C_f$ cannot be of length $3$.  Otherwise, $C_f$ together with the $K_4$ would create a $C_5$ as shown in Figure \ref{fig:lenth3pseudoface}. $C_f$ cannot have length $4$, since together with an interior face sharing only one edge with $C_f$, it would create a $C_5$. $C_f$ cannot have length $5$ since $G$ is $C_5$-free. Therefore, we know that each exterior pseudoface has length at least $6$.
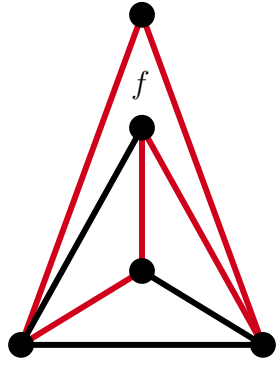
\begin{figure}[H]
    \centering
\tikzset{every picture/.style={line width=0.75pt}} %set default line width to 0.75pt        
\begin{tikzpicture}[x=0.75pt,y=0.75pt,yscale=-1,xscale=1]
%uncomment if require: \path (0,300); %set diagram left start at 0, and has height of 300

%Straight Lines [id:da027130697410049587] 
\draw [color={rgb, 255:red, 208; green, 2; blue, 27 }  ,draw opacity=1 ][line width=2.25]    (320,70) -- (320,142.3) ;
%Straight Lines [id:da914209821398376] 
\draw [color={rgb, 255:red, 208; green, 2; blue, 27 }  ,draw opacity=1 ][line width=2.25]    (320,142.3) -- (259,179.61) ;
%Straight Lines [id:da8052247924950833] 
\draw [line width=2.25]    (381,179.61) -- (320,142.3) ;
%Shape: Circle [id:dp08158551820543014] 
\draw  [fill={rgb, 255:red, 0; green, 0; blue, 0 }  ,fill opacity=1 ] (314,142.3) .. controls (314,138.98) and (316.69,136.3) .. (320,136.3) .. controls (323.31,136.3) and (326,138.98) .. (326,142.3) .. controls (326,145.61) and (323.31,148.3) .. (320,148.3) .. controls (316.69,148.3) and (314,145.61) .. (314,142.3) -- cycle ;
%Straight Lines [id:da9804900536569854] 
\draw [color={rgb, 255:red, 208; green, 2; blue, 27 }  ,draw opacity=1 ][line width=2.25]    (320,13) -- (259,179.61) ;
%Straight Lines [id:da478654968720325] 
\draw [color={rgb, 255:red, 208; green, 2; blue, 27 }  ,draw opacity=1 ][fill={rgb, 255:red, 0; green, 0; blue, 0 }  ,fill opacity=1 ][line width=2.25]    (320,13) -- (381,179.61) ;
%Straight Lines [id:da4433972620905653] 
\draw [line width=2.25]    (320,70) -- (259,179.61) ;
%Straight Lines [id:da2083566223492035] 
\draw [color={rgb, 255:red, 208; green, 2; blue, 27 }  ,draw opacity=1 ][line width=2.25]    (320,70) -- (381,179.61) ;
%Straight Lines [id:da36403505795773583] 
\draw [line width=2.25]    (259,179.61) -- (381,179.61) ;
%Shape: Circle [id:dp4900358209177835] 
\draw  [fill={rgb, 255:red, 0; green, 0; blue, 0 }  ,fill opacity=1 ] (375,179.61) .. controls (375,176.3) and (377.69,173.61) .. (381,173.61) .. controls (384.31,173.61) and (387,176.3) .. (387,179.61) .. controls (387,182.93) and (384.31,185.61) .. (381,185.61) .. controls (377.69,185.61) and (375,182.93) .. (375,179.61) -- cycle ;
%Shape: Circle [id:dp2356146613691712] 
\draw  [fill={rgb, 255:red, 0; green, 0; blue, 0 }  ,fill opacity=1 ] (253,179.61) .. controls (253,176.3) and (255.69,173.61) .. (259,173.61) .. controls (262.31,173.61) and (265,176.3) .. (265,179.61) .. controls (265,182.93) and (262.31,185.61) .. (259,185.61) .. controls (255.69,185.61) and (253,182.93) .. (253,179.61) -- cycle ;
%Shape: Circle [id:dp3756050457797582] 
\draw  [fill={rgb, 255:red, 0; green, 0; blue, 0 }  ,fill opacity=1 ] (314,70) .. controls (314,66.69) and (316.69,64) .. (320,64) .. controls (323.31,64) and (326,66.69) .. (326,70) .. controls (326,73.31) and (323.31,76) .. (320,76) .. controls (316.69,76) and (314,73.31) .. (314,70) -- cycle ;
%Shape: Circle [id:dp23327950659006258] 
\draw  [fill={rgb, 255:red, 0; green, 0; blue, 0 }  ,fill opacity=1 ] (314,13) .. controls (314,9.69) and (316.69,7) .. (320,7) .. controls (323.31,7) and (326,9.69) .. (326,13) .. controls (326,16.31) and (323.31,19) .. (320,19) .. controls (316.69,19) and (314,16.31) .. (314,13) -- cycle ;
% Text Node
\draw (313,40) node [anchor=north west][inner sep=0.75pt]   [align=left] {\Large $f$};
\end{tikzpicture}
    \caption{Exterior pseudoface of $K_4$ of length $3$ would create a $C_5$}
    \label{fig:lenth3pseudoface}
\end{figure}
\end{proof}

Now we show that the following lemma finishes the proof of Theorem \ref{thm:C5Essential}.

\begin{lemma}
\label{lem:C5}
Let $G$ be a $2$-connected $C_5$-free plane graph on $n$ vertices with $\delta(G)\ge 3$. Any triangular block in $G$ satisfies $$
9v(B)-23e(B)+33f(B)\le 0.
$$
\end{lemma}

Once we have proved this, we have $$
9v_G-23e_G+33f_G=\sum_{B\in\mathcal{T}}(9v(B)-23e(B)+33f(B))\le 0.
$$ Combining this with Euler's formula $v_G-e_G+f_G=2$ finishes the proof.

\begin{proof}
We do casework to proceed with the proof. We distinguish the cases according to $B$.

\textbf{Case 1:} $B$ is $K_2$.

Consider the two exterior pseudofaces of $B$. We distinguish two cases.
\begin{enumerate}[(i)]
    \item If there exists one exterior pseudoface $C_f$ of $B$ that is a $4$-face, by Proposition \ref{pro:C5}, we know that any of the triangular blocks having $C_f$ as an exterior face must be a $K_2$. Note that we have assumed $\delta(G)\ge3$, we know that each junction vertex of $B$ is shared by at least three triangular blocks. This gives $v(B)\le 1/3+1/3=2/3$ and note that the other pseudoface has length at least $4$, we have $f(B) \leq 1/4+1/4=1/2$. With $e(B)=1$, we have $9v(B)-23e(B)+33f(B)\le -1/2$.
    \item If both exterior pseudofaces have length greater than $4$, then since $G$ is $C_5$-free, they must have length at least $6$, which gives $f(B)\le 1/6+1/6=1/3$. Since $\delta(G)\ge3$, both vertices of $B$ must be junction vertices, so we have $v(B)\le 1/2+1/2=1$. With $e(B)=1$, we have $9v(B)-23e(B)+33f(B)\le -3$.
\end{enumerate}
%potentialFigs
%\begin{figure}[H]
%    \centering
%    Here we need a figure.
%    \caption{Two different scenarios in Case 1}
%    \label{fig:C5LemmaCase1}
%\end{figure}

\textbf{Case 2:} $B$ is $K_3$.

Since $\delta(G)\ge3$, each vertex of $B$ must be a junction vertex. This gives $v(B)\le 3/2$. Note that $G$ is $C_5$-free and by Proposition \ref{pro:C5}, each exterior pseudoface of $B$ have length at least $6$, so $f(B)\le 1+3/6=3/2$. With $e(B)=3$, we have $9v(B)-23e(B)+33f(B)\le -6$.

\textbf{Case 3:} $B$ is $\Theta_4$.

Note that two of the exterior vertices in $B$ have degree $2$, so they must be junction vertices, which implies $v(B)\le 2+2/2$. By Proposition \ref{pro:C5} we know that the exterior pseudofaces has length at least $6$, which gives $f(B)\le 2+4/6$. With $e(B)=5$, we have $9v(B)-23e(B)+33f(B)\le 0$.

\textbf{Case 4:} $B$ is $K_4$.

Since $G$ is $2$-connected, $B$ has at least two junction vertices, otherwise either $B$ is the whole graph which contradicts the fact that $n\ge 11$, or the unique junction vertex of $B$ would be a cut vertex, a contradiction. Now we consider two different scenarios.

\begin{enumerate}[(i)]
    \item There are $2$ junction vertices in $B$. In this case, we have $v(B)\le 2+2/2$. Note that in this case, two consecutive exterior edges are contained in the boundary of a same exterior face, by Definition \ref{def:triContributionFace} and Proposition \ref{pro:C5} we know that $f(B)\le 3+2/6$. With $e(B)=6$, we have $9v(B)-23e(B)+33f(B)\le -1$.
    \item All $3$ are junction vertices. In this case, we have $v(B)\le 1+3/2$. By Proposition \ref{pro:C5} we know $f(B)\le 3+3/6$. With $e(B)=6$, we have $9v(B)-23e(B)+33f(B)\le 0$.
\end{enumerate}
%potentialFigs
\end{proof}

\section{Quadrangular blocks}
\label{sec:quadrablocks}
\begin{definition}
\label{def:quadrablock}
Let $G$ be a plane graph. For an edge $e\in E(G)$, if it's not contained in any of the $4$-faces of $G$, then we call it a {\bf trivial quadrangular block}. Otherwise, we perform the following algorithm to construct a {\bf quadrangular block} $B$.
\begin{algorithm}[h] 
$B\gets(V(e),e)$\;
    \While{there exists an edge in $B$ such that it is in a bounded $4$-face of $G$ which is not contained in $B$}
    {add all the edges of such bounded $4$-faces to $B$;}
    Output $B$;
\end{algorithm} 

Notice that the resulting quadrangular block does not depend on the choice of starting edge as long as the starting edge is in the quadrangular block.
\end{definition}

\begin{definition}
    \label{def:quadrasettings}
    Let $G$ be a plane graph and $B$ be a quadrangular block.
    \begin{enumerate}
        \item A vertex in $G$ is called a {\bf quadrangular block junction vertex} if it's contained in at least two different quadrangular blocks in $G$.
        \item A vertex or an edge of $B$ is called {\bf quadrangular block exterior} if it lies on the boundary of $B$ and it's called {\bf quadrangular block interior} otherwise. A face is called {\bf quadrangular block exterior} if it is not contained by $B$ while it has at least one common edge with $B$ and it's called {\bf quadrangular block interior} if it's contained by $B$.
    \end{enumerate}
\end{definition}

In the following part of the paper, for the sake of convenience in expression, {\bf quadrangular junction vertex} is abbreviated as {\bf junction vertex}, {\bf quadrangular block exterior} is abbreviated as {\bf exterior} and {\bf quadrangular block interior} is abbreviated as {\bf interior}.

\begin{definition}
\label{def:quadraContribution}
Let $G$ be a plane graph and $B$ is a quadrangular block in $G$. We denote the contribution to the vertex number of a vertex $v$ in $B$ by $n_B(v)$, and define it as$$
n_B(v) = \frac{1}{\#\textit{ of quadrangular blocks in } G \textit{ containing } v}.
$$
The contribution of $B$ to the vertex number is defined as $$v(B)=\sum_{v\in V(B)}n_B(v).$$
\end{definition}

By definition we have $$
v_G=\sum_{v\in V(G)} 1=\sum_{v\in V(G)} \sum_{B\ni v}n_B(v)=\sum_{B\in\mathcal{Q}}\sum_{v\in V(B)}n_B(v)=\sum_{B\in\mathcal{Q}}v(B),
$$ where $\mathcal{Q}$ is the family of all quadrangular blocks in $G$.

We define the contribution of $B$ to the edge number as the number of edges in $B$ and denote it by $e(B)$. Since the quadrangular blocks are edge-disjoint by definition and each edge in $G$ is included in a quadrangular block, we have $$e_G=\sum_{B\in\mathcal{Q}}e(B).$$

Each face in $G$ is either an interior face of a unique quadrangular block or an exterior face of some quadrangular blocks. We denote the interior faces in $G$ by $F_I(G)$ and we denote the exterior faces in $G$ by $F_E(G)$.

\begin{definition}
\label{def:quadraContributionFace}
Let $G$ be a plane graph and $B$ is a quadrangular block in $G$. For each exterior edge $e$ of $B$, we denote its contribution to the face number in $B$ by $f_B(e)$. Note that $e$ can be the edge of at most two exterior faces, we define it as follows.
\begin{enumerate}
    \item If $e$ is contained in the boundary of two exterior faces $f_1$ and $f_2$ (i.e., $e$ is $K_2$ the trivial quadrangular block), then let $f_B(e)=1/l(f_1)+1/l(f_2)$.
    \item Otherwise, let $f_B(e)=1/l(f)$ where $f$ is the exterior face containing $e$.
\end{enumerate}
The contribution of $B$ to the face number is defined as $$f(B)=\#\textit{ interior faces in $B$ } + \sum_{e \textit{ is an exterior edge of $B$ }}f_B(e).$$
\end{definition}

 By definition we have \begin{align*}
    f_G& =\sum_{f\in F(G)}1=\sum_{f\in F_I(G)}1+\sum_{f\in F_E(G)}1
=\sum_{B\in\mathcal{Q}}\#\textit{ interior faces in $B$}
+\sum_{f\in F_E(G)}\frac{l(f)}{l(f)}\\
 & =\sum_{B\in\mathcal{Q}}\#\textit{ interior faces in $B$}
+\sum_{B\in\mathcal{Q}}\sum_{\textit{ $e$ is an exterior edge of $B$}}f_B(e)
=\sum_{B\in\mathcal{Q}}f(B)
 \end{align*}

\section{Quadrangular blocks in bipartite graphs}
\label{sec:bi}

All graphs considered in this section are bipartite planar graphs. Since a bipartite graph does not contain odd cycles, each face is a face with even number of edges.

\subsection{\texorpdfstring{$C_6$}{C6}-free graphs}~

In this section, we prove 
\thmBiCSix*

\subsubsection{Extremal construction showing sharpness}~

The extremal construction below shows that the bound in Theorem \ref{thm:biC6} is sharp.
\begin{figure}[H]
    \centering
    \includegraphics{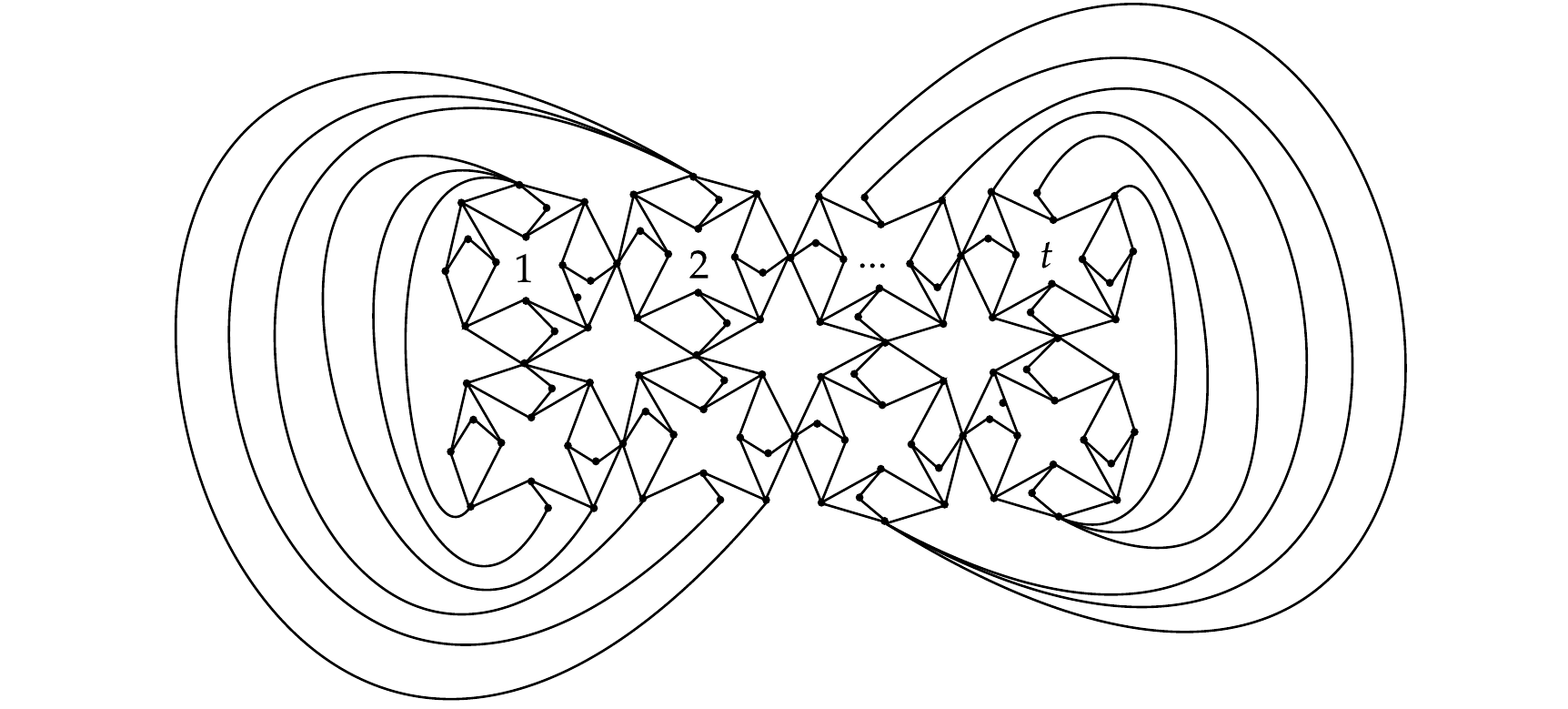}
    \caption{The extremal construction of Theorem \ref{thm:biC6}}
    \label{fig:biC6Extremal}
\end{figure}

In Figure \ref{fig:biC6Extremal}, there are $28t+2$ vertices, $48t$ edges, $8t$ degree $2$ vertices, and $8t+4$ edges joining a degree $2$ vertex and a degree $3$ vertex, which satisfy that $e_G=3v_G/2+k/2+e_{2,3}/4-4$.

\begin{remark}
This construction comes from the proof of Lemma \ref{lem:biC6}. In the extremal graph, each quadrangular block has $0$ ``contribution".
\end{remark}

\subsubsection{Preparatory propositions}

\begin{proposition}
\label{pro:biC6MinDegree}
    Let $G$ be a $C_6$-free planar bipartite graph on $n$ vertices. Then $\delta(G)\le 2$.
\end{proposition}

\begin{proof}
Denote the number of degree $i$ vertices by $n_i$. Suppose the contrary that $\delta(G)\ge3$, we will have $$n=\sum_{i\ge3} n_i$$

We now assign charge $1$ to each face and perform the following discharging method: we discharge $1/l(f)$ to the edges of each face $f$.

Note that $G$ is $C_6$-free and $\delta(G)\ge 3$, we know that $4$-faces cannot be adjacent to each other. So, each edge can receive charge at most $1/4+1/8=3/8$. Consider any degree $3$ vertex $u$, we know that it's not possible for all the three edges incident to $u$ to receive charge $3/8$ as it's not possible for the three edges simultaneously having one side as a $4$-face and an $8$-face on the other side. Therefore, an upper bound of the total charge for three edges incident to a common degree $3$ vertex is $3/8+3/8+1/4$, which means two of the faces between the two edges are $8$-faces and one is a $4$-face.

We double count the total charge. Since each face receives charge $1$ originally, we know that the total charge is $f_G$. On the other hand, after the discharging, for any vertex $v$ having degree $d(v)>3$, the edges incident to it will have total charge at most $d(v)\cdot3/8$ while for any degree $3$ vertex the edges incident to it will have total charge at most $3/8+3/8+1/4=1$. Because each edge is incident to exactly two vertices, take the sum of the upper bound for all vertices we have the total charge is at most $(n_3+3/8\cdot\sum_{i\ge 4} in_i)/2$.

Thus, by Euler's formula we have$$
   4-2\sum_{i\ge3}n_i+\sum_{i\ge3}in_i= 4-2v_G+2e_G=2f_G  \le  n_3+\frac{3}{8}\sum_{i\ge4}in_i,
$$ which implies $$
4\le \frac{3}{8}\sum_{i\ge4}in_i+2\sum_{i\ge4}n_i-\sum_{i\ge4}in_i=\sum_{i\ge4}(2-\frac{5}{8}i)n_i.
$$
While note that $2-i\cdot 5/8<0$ for $i\ge4$, a contradiction. Thus, we know that $\delta(G)\le 2$.
\end{proof}

To prove the theorem, we only need to consider the case that each degree $2$ vertex has at least one neighbor of degree at most $3$. Otherwise, we can perform induction on the number of such vertices by simply deleting one such vertex from the graph. Also, we can suppose that $\delta(G)=2$, because if $\delta(G)=2$, we can recursively eliminate degree $1$ vertices and do induction. We can suppose the graph is connected because otherwise we can consider the inequality on the components separately and sum them up to get an even better bound.

\begin{proposition}
    Let $G$ be a $C_6$-free bipartite plane graph and each degree $2$ vertex of it has at least one neighbor of degree at most $3$. A quadrangular block in $G$ contains at most $5$ vertices.
\end{proposition}

\begin{proof}
For a quadrangular block which contains more than $5$ vertices, it must be initiated from two adjacent $4$-faces by the generating algorithm of quadrangular blocks. If the two $4$-faces share only one edge then this gives a $6$-cycle. If the two $4$-faces share two edges then it is a $K_{2,3}$. Adding another $4$-face to it would create either a $C_6$, or a degree $2$ vertex which has both neighbors of degree more than $3$. We get a contradiction in both cases.
\end{proof}

Now we describe all the possible quadrangular blocks in $G$. As shown in Figure \ref{fig:quadraBlocksBiCSix}, other than the trivial quadrangular block $K_2$, there is one $4$-vertex quadrangular block $C_4$ and one $5$-vertex quadrangular block $K_{2,3}$.

\begin{figure}[H]
    \centering
    \begin{tikzpicture}[x=0.75pt,y=0.75pt,yscale=-1,xscale=1]
%uncomment if require: \path (0,300); %set diagram left start at 0, and has height of 300
%Shape: Square [id:dp8086882513731155] 
\draw  [line width=2.25]  (451.91,104.91) -- (510.5,163.5) -- (451.91,222.09) -- (393.32,163.5) -- cycle ;
%Straight Lines [id:da31242894180587366] 
\draw [line width=2.25]    (393.32,163.5) -- (510.5,163.5) ;
%Shape: Circle [id:dp17319591933431577] 
\draw  [fill={rgb, 255:red, 0; green, 0; blue, 0 }  ,fill opacity=1 ] (387.32,163.5) .. controls (387.32,160.19) and (390.01,157.5) .. (393.32,157.5) .. controls (396.64,157.5) and (399.32,160.19) .. (399.32,163.5) .. controls (399.32,166.81) and (396.64,169.5) .. (393.32,169.5) .. controls (390.01,169.5) and (387.32,166.81) .. (387.32,163.5) -- cycle ;
%Shape: Circle [id:dp6550389249411441] 
\draw  [fill={rgb, 255:red, 0; green, 0; blue, 0 }  ,fill opacity=1 ] (504.5,163.5) .. controls (504.5,160.19) and (507.19,157.5) .. (510.5,157.5) .. controls (513.81,157.5) and (516.5,160.19) .. (516.5,163.5) .. controls (516.5,166.81) and (513.81,169.5) .. (510.5,169.5) .. controls (507.19,169.5) and (504.5,166.81) .. (504.5,163.5) -- cycle ;
%Shape: Circle [id:dp7565652477937277] 
\draw  [fill={rgb, 255:red, 0; green, 0; blue, 0 }  ,fill opacity=1 ] (445.91,222.09) .. controls (445.91,218.77) and (448.6,216.09) .. (451.91,216.09) .. controls (455.23,216.09) and (457.91,218.77) .. (457.91,222.09) .. controls (457.91,225.4) and (455.23,228.09) .. (451.91,228.09) .. controls (448.6,228.09) and (445.91,225.4) .. (445.91,222.09) -- cycle ;
%Shape: Circle [id:dp5383292361241414] 
\draw  [fill={rgb, 255:red, 0; green, 0; blue, 0 }  ,fill opacity=1 ] (445.91,104.91) .. controls (445.91,101.6) and (448.6,98.91) .. (451.91,98.91) .. controls (455.23,98.91) and (457.91,101.6) .. (457.91,104.91) .. controls (457.91,108.23) and (455.23,110.91) .. (451.91,110.91) .. controls (448.6,110.91) and (445.91,108.23) .. (445.91,104.91) -- cycle ;
%Shape: Circle [id:dp06841739190955765] 
\draw  [fill={rgb, 255:red, 0; green, 0; blue, 0 }  ,fill opacity=1 ] (445.91,163.5) .. controls (445.91,160.19) and (448.6,157.5) .. (451.91,157.5) .. controls (455.23,157.5) and (457.91,160.19) .. (457.91,163.5) .. controls (457.91,166.81) and (455.23,169.5) .. (451.91,169.5) .. controls (448.6,169.5) and (445.91,166.81) .. (445.91,163.5) -- cycle ;
%Straight Lines [id:da9233196653601479] 
\draw [line width=2.25]    (47.82,163.5) -- (165,163.5) ;
%Shape: Circle [id:dp16623351208359893] 
\draw  [fill={rgb, 255:red, 0; green, 0; blue, 0 }  ,fill opacity=1 ] (41.82,163.5) .. controls (41.82,160.19) and (44.51,157.5) .. (47.82,157.5) .. controls (51.14,157.5) and (53.82,160.19) .. (53.82,163.5) .. controls (53.82,166.81) and (51.14,169.5) .. (47.82,169.5) .. controls (44.51,169.5) and (41.82,166.81) .. (41.82,163.5) -- cycle ;
%Shape: Circle [id:dp18125163424343982] 
\draw  [fill={rgb, 255:red, 0; green, 0; blue, 0 }  ,fill opacity=1 ] (159,163.5) .. controls (159,160.19) and (161.69,157.5) .. (165,157.5) .. controls (168.31,157.5) and (171,160.19) .. (171,163.5) .. controls (171,166.81) and (168.31,169.5) .. (165,169.5) .. controls (161.69,169.5) and (159,166.81) .. (159,163.5) -- cycle ;
%Shape: Square [id:dp6511525671688458] 
\draw  [line width=2.25]  (277.91,104.91) -- (336.5,163.5) -- (277.91,222.09) -- (219.32,163.5) -- cycle ;
%Shape: Circle [id:dp48606502484821035] 
\draw  [fill={rgb, 255:red, 0; green, 0; blue, 0 }  ,fill opacity=1 ] (213.32,163.5) .. controls (213.32,160.19) and (216.01,157.5) .. (219.32,157.5) .. controls (222.64,157.5) and (225.32,160.19) .. (225.32,163.5) .. controls (225.32,166.81) and (222.64,169.5) .. (219.32,169.5) .. controls (216.01,169.5) and (213.32,166.81) .. (213.32,163.5) -- cycle ;
%Shape: Circle [id:dp7886679666969192] 
\draw  [fill={rgb, 255:red, 0; green, 0; blue, 0 }  ,fill opacity=1 ] (330.5,163.5) .. controls (330.5,160.19) and (333.19,157.5) .. (336.5,157.5) .. controls (339.81,157.5) and (342.5,160.19) .. (342.5,163.5) .. controls (342.5,166.81) and (339.81,169.5) .. (336.5,169.5) .. controls (333.19,169.5) and (330.5,166.81) .. (330.5,163.5) -- cycle ;
%Shape: Circle [id:dp5326048810452135] 
\draw  [fill={rgb, 255:red, 0; green, 0; blue, 0 }  ,fill opacity=1 ] (271.91,222.09) .. controls (271.91,218.77) and (274.6,216.09) .. (277.91,216.09) .. controls (281.23,216.09) and (283.91,218.77) .. (283.91,222.09) .. controls (283.91,225.4) and (281.23,228.09) .. (277.91,228.09) .. controls (274.6,228.09) and (271.91,225.4) .. (271.91,222.09) -- cycle ;
%Shape: Circle [id:dp3386490452384914] 
\draw  [fill={rgb, 255:red, 0; green, 0; blue, 0 }  ,fill opacity=1 ] (271.91,104.91) .. controls (271.91,101.6) and (274.6,98.91) .. (277.91,98.91) .. controls (281.23,98.91) and (283.91,101.6) .. (283.91,104.91) .. controls (283.91,108.23) and (281.23,110.91) .. (277.91,110.91) .. controls (274.6,110.91) and (271.91,108.23) .. (271.91,104.91) -- cycle ;

% Text Node
\draw (436,238) node [anchor=north west][inner sep=0.75pt]  [font=\Large] [align=left] {$K_{2,3}$};
% Text Node
\draw (92,238) node [anchor=north west][inner sep=0.75pt]  [font=\Large] [align=left] {$K_2$};
% Text Node
\draw (267,238) node [anchor=north west][inner sep=0.75pt]  [font=\Large] [align=left] {$C_4$};
\end{tikzpicture}
    \caption{All the possible quadrangular blocks in $G$}
    \label{fig:quadraBlocksBiCSix}
\end{figure}
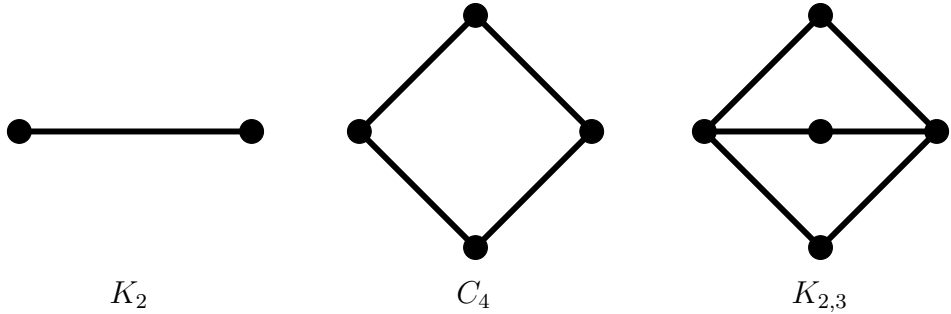

\begin{definition}
    \label{def:BiC6QuadraContribution}
    Let $G$ be a $C_6$-free bipartite plane graph and each degree $2$ vertex of it has at least one neighbor of degree at most $3$, $B$ is a quadrangular block in $G$. We denote the contribution of $v$ to the number of degree $2$ vertices by $k_B(v)$, and define it as $$
    k_B(v) = \begin{cases}
        1/ \#\textit{ quadrangular blocks in } G \textit{ containing } v & d(v)= 2\\
        0 & d(v)\neq 2
    \end{cases}.
    $$ The contribution of $B$ to the number of degree $2$ vertices is defines as $$
    k(B)=\sum_{v\in V(B)}k_B(v).
    $$
\end{definition}
By definition we have $$
k=\sum_{v\in V(G)}1=\sum_{v\in V(G)}\sum_{B\ni v}k_B(v)=\sum_{B\in \mathcal{Q}}\sum_{v\in V(B)}k_B(v)=\sum_{B\in\mathcal{Q}}k(B), 
$$ where $k$ is the number of degree $2$ vertices in $G$.

We define the contribution of $B$ to $e_{2,3}$ as the number of edges in $B$ such that one vertex of the edge has degree $2$ and the other has degree $3$ and we denote the contribution by $e_{2,3}(B)$. Since the quadrangular blocks form a partition of the edge set of $G$, we have $$
e_{2,3} = \sum_{B\in\mathcal{Q}}e_{2,3}(B),
$$ where $e_{2,3}$ is the number of edges in $G$ such that one vertex of the edge has degree $2$ and the other has degree $3$.

The following proposition is useful in the proof:
\begin{proposition}
    \label{pro:biC6}
    Let $G$ be a $C_6$-free bipartite plane graph with $\delta(G)=2$, then each exterior face of any quadrangular block $B$ in $G$ has length at least $8$. 
\end{proposition}
\begin{proof}
    Since $G$ is bipartite, $G$ only has even faces. By definition of quadrangular blocks we know that exterior faces cannot be $4$-faces. Since $G$ is $C_6$-free, we deduce that each exterior face has length at least $8$.
\end{proof}

\subsubsection{Proof of Theorem \ref{thm:biC6}}~

Our main target here is to show that 
\begin{lemma}
\label{lem:biC6}
    Let $G$ be a $C_6$-free bipartite plane graph on $n$ vertices, $\delta(G)=2$, and each degree $2$ vertex of it has at least one neighbor of degree at most $3$. Let $k$ be the number of degree $2$ vertices in $G$ and $e_{2,3}$ be the number of edges in $G$ such that one vertex of the edge has degree $2$ and the other has degree $3$. Any quadrangular block $B$ in $G$ satisfies$$
    2v(B)-4e(B)+8f(B)-2k(B)-e_{2,3}(B)\le 0.
    $$
\end{lemma} Once we have proved this, we have $$
2v_G-4e_G+8f_G-2k-e_{2,3} = \sum_{B\in\mathcal{Q}}(2v(B)-4e(B)+8f(B)-2k(B)-e_{2,3}(B))\le 0.
$$ Combining this with Euler's formula $v_G-e_G+f_G=2$ finishes the proof.

\begin{proof}
We do casework to proceed with the proof. We distinguish the cases according to $B$.

\textbf{Case 1:} $B$ is $K_2$.

Since $\delta(G)=2$, each vertex of $B$ must be a junction vertex. This gives $v(B)\le 1/2+1/2=1$. By Proposition \ref{pro:biC6}, we know that $f(B)\le 1/8+1/8$ since each exterior face has length at least $8$. With $e(B)=1$, we have $2v(B)-4e(B)+8f(B)-2k(B)-e_{2,3}(B)\le 2v(B)-4e(B)+8f(B)\le 0$.

\textbf{Case 2:} $B$ is $C_4$.

Since $G$ is connected, we know that there is at least one junction vertex in $B$ other than the trivial case that $C_4$ is the whole graph. By Proposition \ref{pro:biC6}, we always have $f(B)\le 1+4/8$. There are four different scenarios.
\begin{enumerate}[(i)]
    \item There's $1$ junction vertex in $B$. In this case, we have $v(B)\le 3+1/2$. Since the other three exterior vertices are not junction vertices, we know that they are of degree $2$, so $k(B)=3$. With $e(B)=4$ and $f(B)\le 1+4/8$, we have $2v(B)-4e(B)+8f(B)-2k(B)-e_{2,3}(B)\le 2v(B)-4e(B)+8f(B)-2k(B) \le -3$.
    \item There are $2$ junction vertices in $B$. In this case, we have $v(B)\le 2+2/2$. There are two vertices of degree $2$ so $k(B)=2$. With $e(B)=4$ and $f(B)\le 1+4/8$, we have $2v(B)-4e(B)+8f(B)-2k(B)-e_{2,3}(B)\le 2v(B)-4e(B)+8f(B)-2k(B) \le -2$.
    \item There are $3$ junction vertices in $B$ so $v(B)\le 1+3/2$. There is one vertex of degree $2$ so $k(B)=1$. With $e(B)=4$ and $f(B)\le 1+4/8$, we have $2v(B)-4e(B)+8f(B)-2k(B)-e_{2,3}(B)\le 2v(B)-4e(B)+8f(B)-2k(B) \le -1$.
    \item There are $4$ junction vertices in $B$ so $v(B)\le 4/2$ and $k(B)=0$. With $e(B)=4$ and $f(B)\le 1+4/8$, we have $2v(B)-4e(B)+8f(B)-2k(B)-e_{2,3}(B)\le 2v(B)-4e(B)+8f(B)-2k(B) \le 0$.
\end{enumerate}
%potentialFigs

\textbf{Case 3:} $B$ is $K_{2,3}$.

Since $G$ is connected, we know that there is at least one junction vertex in $B$ other than the trivial case that $K_{2,3}$ is the whole graph, which would not be the case since we are considering a graph on more than $5$ vertices. Therefore, by Proposition \ref{pro:biC6}, we always have $f(B)\le 2+4/8$ and $e(B)=6$. By assumption, the two exterior vertices adjacent to the degree $2$ interior vertex cannot both be junction vertices, so there are four different scenarios.
\begin{enumerate}[(i)]
    \item Neither of the two degree $2$ exterior vertex in $B$ is a junction vertex. In this case, since there is at least one junction vertex, we have $v(B)\le 4+1/2$ and $k(B)=3$. Thus, we have $2v(B)-4e(B)+8f(B)-2k(B)-e_{2,3}(B)\le 2v(B)-4e(B)+8f(B)-2k(B) \le -1$.
    \item One of the two degree $2$ exterior vertices in $B$ is a junction vertex. In this case, $v(B)\le 4+1/2$ and $k(B)\ge 2$. Note that at least one of the two degree $3$ exterior vertices is not a junction vertex, we also have $e_{2,3}(B)\ge 2$. Thus, we have $2v(B)-4e(B)+8f(B)-2k(B)-e_{2,3}(B)\le -1$.
    \item Both of the two degree $2$ exterior vertices in $B$ are junction vertices, and $B$ has $2$ junction vertices in all. In this case, we have $v(B) \le 3+2/2$ and $k(B)=1$, $e_{2,3}(B)=2$. Thus, we have $2v(B)-4e(B)+8f(B)-2k(B)-e_{2,3}(B)\le 0$.
    \item Both of the two degree $2$ exterior vertices in $B$ are junction vertices, and $B$ has $3$ junction vertices in all. In this case, we have $v(B) \le 2+3/2$ and $k(B)=1$, $e_{2,3}(B)=1$. Thus, we have $2v(B)-4e(B)+8f(B)-2k(B)-e_{2,3}(B)\le 0$.
\end{enumerate}
%potentialFigs

\end{proof}

\subsection{\texorpdfstring{$C_8$}{C8}-free graphs}~

In this section, we prove
\thmBiCEight*

\subsubsection{Extremal construction showing sharpness}~

The extremal construction below shows that the bound in Theorem \ref{thm:biC8} is sharp.
\begin{figure}[H]
    \centering
    \includegraphics{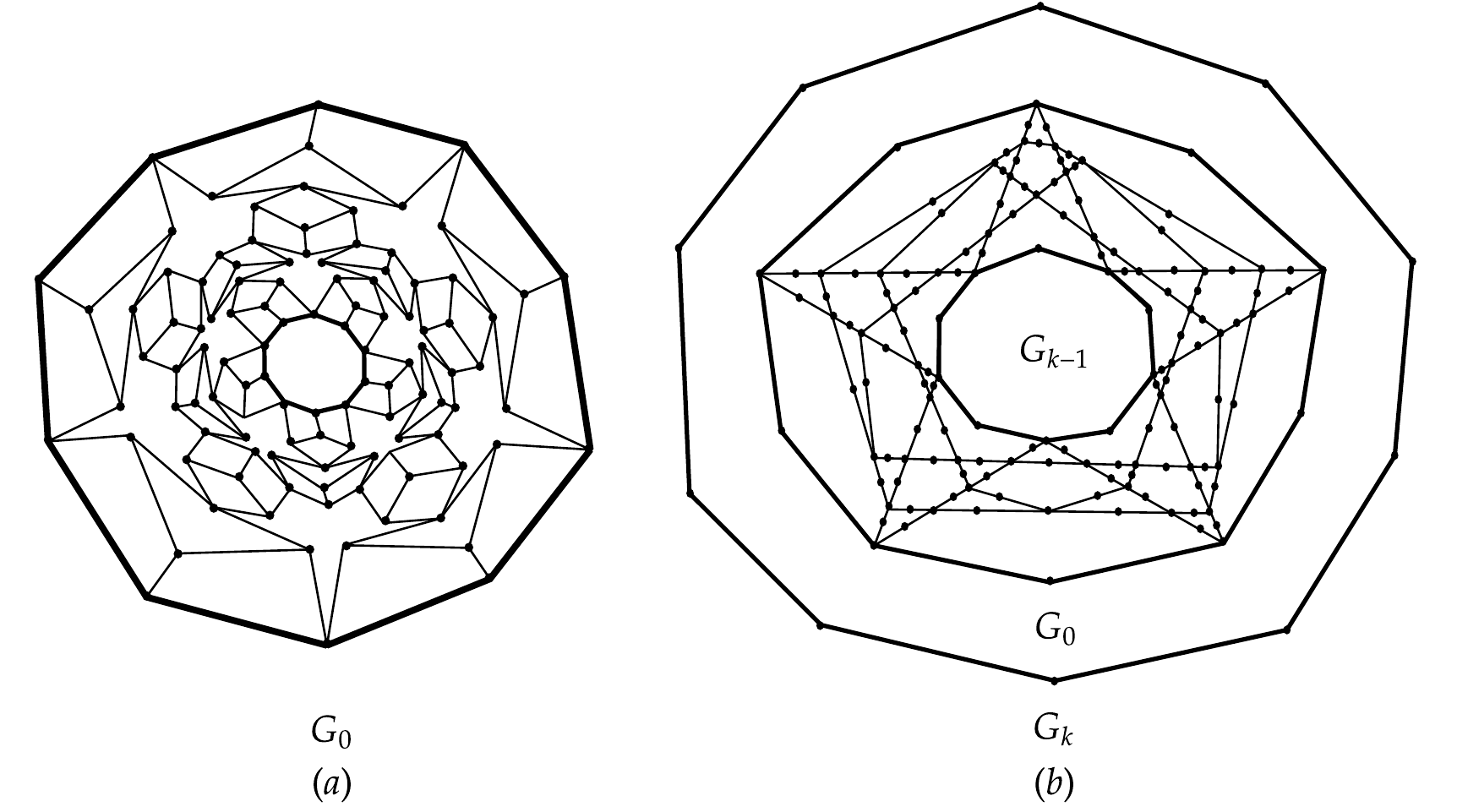}
    \caption{Construction of $G_k$ that gives the extremal construction of Theorem \ref{thm:biC8}}
    \label{fig:biC8Extremal}
\end{figure}

As shown in Figure \ref{fig:biC8Extremal}, we construct an extremal plane graph satisfying the assumptions in Theorem \ref{thm:biC8}. Let $G_0$ be the graph depicted in Figure \ref{fig:biC8Extremal}(a), and in Figure \ref{fig:biC8Extremal}(b), each face of length $6$ is actually a $Q_7$ quadrangular block as shown in Figure \ref{fig:quadraBlocksBiCEight}. We don't draw the interior details for readability considerations.

We construct $G_k$ recursively for any positive integer $k$ via the illustration given in Figure \ref{fig:biC8Extremal}(b): the entire graph $G_{k-1}$ is placed into the central decagon of Figure \ref{fig:biC8Extremal}(b), and the entire $G_0$ is then placed between the two given bold decagons of Figure \ref{fig:biC8Extremal}(b), in such a way that these are identified with the bold decagons in Figure \ref{fig:biC8Extremal}(a). One can check that $G_k$ is bipartite and $C_8$-free with $270k+110$ vertices and  $450k+180$ edges, which satisfy that $e_G=5(v_G-2)/3$.

\begin{remark}
This construction comes from the proof of Lemma \ref{lem:biC8}. In the extremal graph, each quadrangular block has $0$ ``contribution".
\end{remark}

\subsubsection{Preparatory propositions}

\begin{proposition}
\label{pro:biC8Block}
    Let $G$ be a $C_8$-free bipartite plane graph on $n$ vertices with $\delta(G)\ge 3$. A quadrangular block in $G$ contains at most $7$ vertices.
\end{proposition}
\begin{proof}
    For a quadrangular block which contains more than $7$ vertices, it must be initiated from two adjacent $4$-faces sharing only one common edge since $\delta(G)\ge 3$. Adding another $4$ face would create a $C_8$ or a degree $2$ vertex with the only exception that the three $4$-faces pairwise share exact one common edge, resulting in a quadrangular block with $7$ vertices. Adding another $4$-face to it in an arbitrary way would create degree $2$ vertex or a $C_8$.
\end{proof}

Now we describe all the possible quadrangular blocks in $G$. As shown in Figure \ref{fig:quadraBlocksBiCEight}, other than the trivial quadrangular block $K_2$, there is one $4$-vertex quadrangular block $C_4$, one $6$-vertex quadrangular block $\Theta_6$,  and one $7$-vertex quadrangular block that is denoted by $Q_7$.

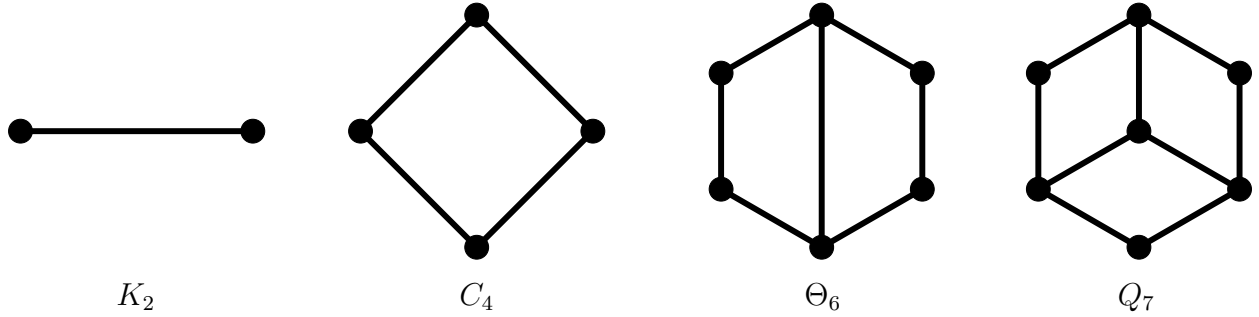
\begin{figure}[H]
    \centering
\begin{tikzpicture}[x=0.75pt,y=0.75pt,yscale=-1,xscale=1]
%uncomment if require: \path (0,300); %set diagram left start at 0, and has height of 300

%Straight Lines [id:da11987314299222929] 
\draw [line width=2.25]    (431.91,84.91) -- (431.91,202.09) ;
%Shape: Circle [id:dp975250785945412] 
\draw  [fill={rgb, 255:red, 0; green, 0; blue, 0 }  ,fill opacity=1 ] (585.91,143.5) .. controls (585.91,140.19) and (588.6,137.5) .. (591.91,137.5) .. controls (595.23,137.5) and (597.91,140.19) .. (597.91,143.5) .. controls (597.91,146.81) and (595.23,149.5) .. (591.91,149.5) .. controls (588.6,149.5) and (585.91,146.81) .. (585.91,143.5) -- cycle ;
%Shape: Circle [id:dp2165776269609594] 
\draw  [fill={rgb, 255:red, 0; green, 0; blue, 0 }  ,fill opacity=1 ] (375.17,114.21) .. controls (375.17,110.89) and (377.86,108.21) .. (381.17,108.21) .. controls (384.49,108.21) and (387.17,110.89) .. (387.17,114.21) .. controls (387.17,117.52) and (384.49,120.21) .. (381.17,120.21) .. controls (377.86,120.21) and (375.17,117.52) .. (375.17,114.21) -- cycle ;
%Shape: Circle [id:dp9843406924738909] 
\draw  [fill={rgb, 255:red, 0; green, 0; blue, 0 }  ,fill opacity=1 ] (476.65,172.79) .. controls (476.65,169.48) and (479.34,166.79) .. (482.65,166.79) .. controls (485.96,166.79) and (488.65,169.48) .. (488.65,172.79) .. controls (488.65,176.11) and (485.96,178.79) .. (482.65,178.79) .. controls (479.34,178.79) and (476.65,176.11) .. (476.65,172.79) -- cycle ;
%Shape: Circle [id:dp8008159406276838] 
\draw  [fill={rgb, 255:red, 0; green, 0; blue, 0 }  ,fill opacity=1 ] (476.65,114.21) .. controls (476.65,110.89) and (479.34,108.21) .. (482.65,108.21) .. controls (485.96,108.21) and (488.65,110.89) .. (488.65,114.21) .. controls (488.65,117.52) and (485.96,120.21) .. (482.65,120.21) .. controls (479.34,120.21) and (476.65,117.52) .. (476.65,114.21) -- cycle ;
%Shape: Circle [id:dp47502628685617565] 
\draw  [fill={rgb, 255:red, 0; green, 0; blue, 0 }  ,fill opacity=1 ] (425.91,84.91) .. controls (425.91,81.6) and (428.6,78.91) .. (431.91,78.91) .. controls (435.23,78.91) and (437.91,81.6) .. (437.91,84.91) .. controls (437.91,88.23) and (435.23,90.91) .. (431.91,90.91) .. controls (428.6,90.91) and (425.91,88.23) .. (425.91,84.91) -- cycle ;
%Straight Lines [id:da11468656816094547] 
\draw [line width=2.25]    (27.82,143.5) -- (145,143.5) ;
%Shape: Circle [id:dp0154309753346622] 
\draw  [fill={rgb, 255:red, 0; green, 0; blue, 0 }  ,fill opacity=1 ] (21.82,143.5) .. controls (21.82,140.19) and (24.51,137.5) .. (27.82,137.5) .. controls (31.14,137.5) and (33.82,140.19) .. (33.82,143.5) .. controls (33.82,146.81) and (31.14,149.5) .. (27.82,149.5) .. controls (24.51,149.5) and (21.82,146.81) .. (21.82,143.5) -- cycle ;
%Shape: Circle [id:dp861593456120753] 
\draw  [fill={rgb, 255:red, 0; green, 0; blue, 0 }  ,fill opacity=1 ] (139,143.5) .. controls (139,140.19) and (141.69,137.5) .. (145,137.5) .. controls (148.31,137.5) and (151,140.19) .. (151,143.5) .. controls (151,146.81) and (148.31,149.5) .. (145,149.5) .. controls (141.69,149.5) and (139,146.81) .. (139,143.5) -- cycle ;
%Shape: Regular Polygon [id:dp470777050189477] 
\draw  [line width=2.25]  (431.91,202.09) -- (381.17,172.79) -- (381.17,114.21) -- (431.91,84.91) -- (482.65,114.21) -- (482.65,172.79) -- cycle ;
%Shape: Circle [id:dp6526240060547581] 
\draw  [fill={rgb, 255:red, 0; green, 0; blue, 0 }  ,fill opacity=1 ] (375.17,172.79) .. controls (375.17,169.48) and (377.86,166.79) .. (381.17,166.79) .. controls (384.49,166.79) and (387.17,169.48) .. (387.17,172.79) .. controls (387.17,176.11) and (384.49,178.79) .. (381.17,178.79) .. controls (377.86,178.79) and (375.17,176.11) .. (375.17,172.79) -- cycle ;
%Shape: Circle [id:dp45329268817311497] 
\draw  [fill={rgb, 255:red, 0; green, 0; blue, 0 }  ,fill opacity=1 ] (425.91,202.09) .. controls (425.91,198.77) and (428.6,196.09) .. (431.91,196.09) .. controls (435.23,196.09) and (437.91,198.77) .. (437.91,202.09) .. controls (437.91,205.4) and (435.23,208.09) .. (431.91,208.09) .. controls (428.6,208.09) and (425.91,205.4) .. (425.91,202.09) -- cycle ;
%Shape: Circle [id:dp014609434012019529] 
\draw  [fill={rgb, 255:red, 0; green, 0; blue, 0 }  ,fill opacity=1 ] (535.17,114.21) .. controls (535.17,110.89) and (537.86,108.21) .. (541.17,108.21) .. controls (544.49,108.21) and (547.17,110.89) .. (547.17,114.21) .. controls (547.17,117.52) and (544.49,120.21) .. (541.17,120.21) .. controls (537.86,120.21) and (535.17,117.52) .. (535.17,114.21) -- cycle ;
%Shape: Circle [id:dp6078909543158129] 
\draw  [fill={rgb, 255:red, 0; green, 0; blue, 0 }  ,fill opacity=1 ] (636.65,172.79) .. controls (636.65,169.48) and (639.34,166.79) .. (642.65,166.79) .. controls (645.96,166.79) and (648.65,169.48) .. (648.65,172.79) .. controls (648.65,176.11) and (645.96,178.79) .. (642.65,178.79) .. controls (639.34,178.79) and (636.65,176.11) .. (636.65,172.79) -- cycle ;
%Shape: Circle [id:dp5983804469232601] 
\draw  [fill={rgb, 255:red, 0; green, 0; blue, 0 }  ,fill opacity=1 ] (636.65,114.21) .. controls (636.65,110.89) and (639.34,108.21) .. (642.65,108.21) .. controls (645.96,108.21) and (648.65,110.89) .. (648.65,114.21) .. controls (648.65,117.52) and (645.96,120.21) .. (642.65,120.21) .. controls (639.34,120.21) and (636.65,117.52) .. (636.65,114.21) -- cycle ;
%Shape: Circle [id:dp9407569106637883] 
\draw  [fill={rgb, 255:red, 0; green, 0; blue, 0 }  ,fill opacity=1 ] (585.91,84.91) .. controls (585.91,81.6) and (588.6,78.91) .. (591.91,78.91) .. controls (595.23,78.91) and (597.91,81.6) .. (597.91,84.91) .. controls (597.91,88.23) and (595.23,90.91) .. (591.91,90.91) .. controls (588.6,90.91) and (585.91,88.23) .. (585.91,84.91) -- cycle ;
%Shape: Regular Polygon [id:dp6163269226629693] 
\draw  [line width=2.25]  (591.91,202.09) -- (541.17,172.79) -- (541.17,114.21) -- (591.91,84.91) -- (642.65,114.21) -- (642.65,172.79) -- cycle ;
%Shape: Circle [id:dp1895153958299296] 
\draw  [fill={rgb, 255:red, 0; green, 0; blue, 0 }  ,fill opacity=1 ] (535.17,172.79) .. controls (535.17,169.48) and (537.86,166.79) .. (541.17,166.79) .. controls (544.49,166.79) and (547.17,169.48) .. (547.17,172.79) .. controls (547.17,176.11) and (544.49,178.79) .. (541.17,178.79) .. controls (537.86,178.79) and (535.17,176.11) .. (535.17,172.79) -- cycle ;
%Shape: Circle [id:dp6326946079418425] 
\draw  [fill={rgb, 255:red, 0; green, 0; blue, 0 }  ,fill opacity=1 ] (585.91,202.09) .. controls (585.91,198.77) and (588.6,196.09) .. (591.91,196.09) .. controls (595.23,196.09) and (597.91,198.77) .. (597.91,202.09) .. controls (597.91,205.4) and (595.23,208.09) .. (591.91,208.09) .. controls (588.6,208.09) and (585.91,205.4) .. (585.91,202.09) -- cycle ;
%Straight Lines [id:da5663609511704109] 
\draw [line width=2.25]    (591.91,84.91) -- (591.91,143.5) ;
%Straight Lines [id:da6055731372971558] 
\draw [line width=2.25]    (591.91,143.5) -- (541.17,172.79) ;
%Straight Lines [id:da3220701634580381] 
\draw [line width=2.25]    (591.91,143.5) -- (642.65,172.79) ;
%Shape: Square [id:dp31398710544055564] 
\draw  [line width=2.25]  (257.91,84.91) -- (316.5,143.5) -- (257.91,202.09) -- (199.32,143.5) -- cycle ;
%Shape: Circle [id:dp5209260697323468] 
\draw  [fill={rgb, 255:red, 0; green, 0; blue, 0 }  ,fill opacity=1 ] (193.32,143.5) .. controls (193.32,140.19) and (196.01,137.5) .. (199.32,137.5) .. controls (202.64,137.5) and (205.32,140.19) .. (205.32,143.5) .. controls (205.32,146.81) and (202.64,149.5) .. (199.32,149.5) .. controls (196.01,149.5) and (193.32,146.81) .. (193.32,143.5) -- cycle ;
%Shape: Circle [id:dp7595450927069378] 
\draw  [fill={rgb, 255:red, 0; green, 0; blue, 0 }  ,fill opacity=1 ] (310.5,143.5) .. controls (310.5,140.19) and (313.19,137.5) .. (316.5,137.5) .. controls (319.81,137.5) and (322.5,140.19) .. (322.5,143.5) .. controls (322.5,146.81) and (319.81,149.5) .. (316.5,149.5) .. controls (313.19,149.5) and (310.5,146.81) .. (310.5,143.5) -- cycle ;
%Shape: Circle [id:dp12027768505551384] 
\draw  [fill={rgb, 255:red, 0; green, 0; blue, 0 }  ,fill opacity=1 ] (251.91,202.09) .. controls (251.91,198.77) and (254.6,196.09) .. (257.91,196.09) .. controls (261.23,196.09) and (263.91,198.77) .. (263.91,202.09) .. controls (263.91,205.4) and (261.23,208.09) .. (257.91,208.09) .. controls (254.6,208.09) and (251.91,205.4) .. (251.91,202.09) -- cycle ;
%Shape: Circle [id:dp2889193861108015] 
\draw  [fill={rgb, 255:red, 0; green, 0; blue, 0 }  ,fill opacity=1 ] (251.91,84.91) .. controls (251.91,81.6) and (254.6,78.91) .. (257.91,78.91) .. controls (261.23,78.91) and (263.91,81.6) .. (263.91,84.91) .. controls (263.91,88.23) and (261.23,90.91) .. (257.91,90.91) .. controls (254.6,90.91) and (251.91,88.23) .. (251.91,84.91) -- cycle ;

% Text Node
\draw (75,219) node [anchor=north west][inner sep=0.75pt]  [font=\Large] [align=left] {$K_2$};
% Text Node
\draw (248,219) node [anchor=north west][inner sep=0.75pt]  [font=\Large] [align=left] {$C_4$};
% Text Node
\draw (422,219) node [anchor=north west][inner sep=0.75pt]  [font=\Large] [align=left] {$\Theta_6$};
% Text Node
\draw (580,219) node [anchor=north west][inner sep=0.75pt]  [font=\Large] [align=left] {$Q_7$};
\end{tikzpicture}
    \caption{All the possible quadrangular blocks in $G$}
    \label{fig:quadraBlocksBiCEight}
\end{figure}

\begin{proposition}
    \label{pro:biC8}
    Let $G$ be a $C_8$-free bipartite plane graph on $n$ vertices with $\delta(G)\ge 3$.
    \begin{enumerate}
        \item If $B$ is a nontrivial quadrangular block, and its exterior face $f$ has only one common edge with an interior face of $B$, then $f$ has length at least $10$. 
        \item If $B$ is $\Theta_6$, and it shares two disjoint exterior edges with an exterior face $f$, then $f$ has length at least $10$.
        \item If $B$ is $\Theta_6$, and it shares three edges with an exterior face $f$, in which two of them are consecutive, then $f$ has length at least $10$.
        \item If $B$ is $\Theta_6$, and it shares two disjoint pairs of consecutive edges with an exterior face $f$, then $f$ has length at least $10$.
        \item If $B$ is $\Theta_6$ or $Q_7$, and it shares three pairwise disjoint edges with an exterior face $f$, then $f$ has length at least $10$.
        % this has nothing to do with Theta_6, for Theta_6, we only use (1).
    \end{enumerate}
\end{proposition}
\begin{proof}~
    \begin{enumerate}
        \item The exterior face $f$ of $B$ cannot have length $4$ by definition \ref{def:quadrablock}.  $f$ cannot have length $6$ because together with its adjacent interior face this would give a $C_8$, and $f$ cannot have length $8$ because $G$ is $C_8$-free. Since $G$ is bipartite and only contains even faces, we know that $f$ has length at least $10$.
        \item If the two disjoint exterior edges shared by $B$ and its exterior face $f$ are from two different interior faces, then the conclusion follows from (1). Otherwise, suppose the two disjoint exterior edges are from the interior face $f_1$, then to avoid multiple edges and making odd cycles, $f$ should have length at least $2+2\times 3 = 8$. Since $G$ is $C_8$-free, we know that $f$ has length at least $10$.
        \item Since $\delta(G)\ge 3$, the pair of consecutive edges belongs to different interior faces. To avoid making odd cycles, $f$ should have length at least $3+3+2=8$. Since $G$ is $C_8$-free, we know that $f$ has length at least $10$.
        \item Since $\delta(G)\ge 3$, both of the two pairs of consecutive edges belong to different interior faces. To avoid making multi-edges or odd cycles, $f$ should have length at least $4+2\times 3=10$.
        \item To avoid multi-edges, $f$ should have length at least $3+3\times 2=9$. Note that $G$ is a bipartite graph, so $f$ has length at least $10$. 
    \end{enumerate}
\end{proof}

\subsubsection{Proof of Theorem \ref{thm:biC8}}~

Our goal is to show that 
\begin{lemma}
\label{lem:biC8}
    Let $G$ be a $C_8$-free bipartite plane graph on $n$ vertices with $\delta(G)\ge 3$. Any quadrangular block $B$ in $G$ satisfies $$
    5f(B)-2e(B)\le 0.
    $$
\end{lemma}

Once we have proved this, we have $$
5f_G-2e_G = \sum_{B\in\mathcal{Q}}(5f(B)-2e(B)) \le 0.
$$ Combining this with Euler's formula $v_G-e_G+f_G = 2$ finishes the proof.

\begin{proof}
    We do casework to proceed with the proof. We distinguish four cases according to $B$.

    \textbf{Case 1}: $B$ is $K_2$.

    Since $G$ is a bipartite planar graph, each face has even length. By definition of quadrangular blocks, we know that each exterior face has length at least $6$, so $f(B)\le 1/6+1/6$. With $e(B)=1$, we have $5f(B)-2e(B)\le -1/3$.

    \textbf{Case 2}: $B$ is $C_4$.

    Since $\delta(G)\ge 3$ and each vertex in $B$ is of degree $2$, we know that all of them are junction vertices. So $B$ has at least $3$ exterior faces and at least two of them share only one edge with $B$, so $f(B)\le 1+2/10+2/6$ by Proposition \ref{pro:biC8}. With $e(B)=4$, we have $5f(B)-2e(B) \le -1/3$.

    \textbf{Case 3}: $B$ is $\Theta_6$.

    There are four degree $2$ vertices in $B$, we know that they are junction vertices. So there are at least two edges in $B$ such that each of them is the only edge shared by an interior face of $B$ and an exterior face of $B$. There are three different scenarios.

    %\textit{\color{blue} xianzhi: this sentence needs rewording, also, prop 4.7 didn't mention theta 6}
    \begin{enumerate}[(i)]
        \item There are only four junction vertices. In this case, for any exterior face $f$ of $B$, the common edges of $f$ and $B$ can be one single edge, two disjoint single edges, a pair of consecutive edges or two pairs of consecutive edges. In any case, by Proposition \ref{pro:biC8}, it is easy to check that this exterior face hase length at lesat $10$. Therefore, $f(B)\le 2+6/10$. With $e(B)=7$, we have $5f(B)-2e(B)\le -1$.
        %\item There are only four junction vertices. In this case, there are two pairs of consecutive exterior edges of $B$, each lying in a same exterior face. By Proposition \ref{pro:biC8}, we know that if the two pairs lie in different exterior faces, then each of the two exterior faces share a single edge with an interior face of $B$, so we always know that any exterior face of $B$ has length at least $10$. If the two pairs lie in a same exterior face, then this face cannot have length $4$ by definition of quadrangular blocks, cannot have length $6$ because otherwise a multiple edge shall appear. So still we have that each exterior face has length at least $10$ since $G$ is $C_8$-free. Therefore, $f(B)\le 2+6/10$. With $e(B)=7$, we have $5f(B)-2e(B)\le -1$.
        \item There are five junction vertices. In this case, for any exterior face $f$ of $B$, the common edges of $f$ and $B$ can be one single edge, two disjoint single edges, a pair of consecutive edges or a pair of consecutive edges together with a disjoint single edge. In any case, by Proposition \ref{pro:biC8}, it is easy to check that this exterior face hase length at lesat $10$. Therefore, $f(B)\le 2+6/10$. With $e(B)=7$, we have $5f(B)-2e(B)\le -1$.
        %\item There are five junction vertices. In this case, there are at least three edges in $B$ such that each of them is the only edge shared by an interior face of $B$ and an exterior face of $B$. Now consider the remaining three exterior edges other than the edges in those three exterior faces. They are a pair of consecutive exterior face and another exterior edge. If they belong to different exterior faces, by Proposition \ref{pro:biC8} we know that each exterior face of $B$ has length at least $10$. If they belong to a same face, then this face cannot have length $4$ by definition of quadrangular blocks, cannot have length $6$ because otherwise an odd cycle shall appear in $G$ but $G$ is bipartite. So still we have that each exterior face has length at least $10$. Therefore, $f(B)\le 2+6/10$. With $e(B)=7$, we have $5f(B)-2e(B)\le -1$.
        \item There are six junction vertices. In this case, for any exterior faces $f$ of $B$, the common edges of $f$ and $B$ can be one single edge, two disjoint single edges or three pairwise disjoint single edges. In any case by Proposition \ref{pro:biC8}, it is easy to check that this exterior face has length at least $10$, so $f(B)\le 2+6/10$. Together with $e(B)=7$, we have $5f(B)-2e(B)\le -1$.
        %\item There are six junction vertices. In this case, each edge is the only edge shared by an interior face of $B$ and an exterior face. So each exterior face has length at least $10$ by Proposition \ref{pro:biC8}. This gives $f(B)\le 2+6/10$. Together with $e(B)=7$, we have $5f(B)-2e(B)\le -1$.
    \end{enumerate}

    \textbf{Case 4}: $B$ is $Q_7$.
    
    There are three degree $2$ vertices in $B$ and they have to be junction vertices. Therefore, for any exterior faces $f$ of $B$, the common edges of $f$ and $B$ can be one single edge, a pair of consecutive edges, two disjoint single edges, a pair of consecutive edges and a disjoint single edge, three pairwise disjoint single edges. In any case by Proposition \ref{pro:biC8}, it is easy to check that this exterior face has length at least $10$. So $f(B)\le 3+6/10$, and with $e(B)=9$, we have $5f(B)-2e(B)\le 0$.

\end{proof}

\begin{remark}
Note that in this case, $v(B)$ does not appear in the block-wise contribution formula, and this indicates that the contribution method used here is actually a rewording of the classic discharging method in planar graph theory \cite{west2020combinatorial}. One can give a proof using discharging method by assigning an unbalanced discharging according to different quadrangular blocks. Thus, in some sense, the contribution method is an extension of the classic discharging method.
\end{remark}

\subsection{\texorpdfstring{$C_8, C_{10}$}{C8, C10}-free graphs}~

By far, using quadrangular blocks, we have studied bipartite planar graphs while forbidding a small cycle to appear as a subgraph. We've also found an interesting result on simultaneously forbidding two small cycles to appear in a bipartite planar graph, which might provide some inspirations for future studies in this field.

Note that in Proposition \ref{pro:biC6MinDegree} we've already showed that for $C_6$-free bipartite planar graphs the minimum degree is at most $2$, hence we prefer to forbid cycles of length larger than $6$ from bipartite planar graphs.
\thmBiCEightCTen*

\subsubsection{Extremal construction showing sharpness}~

The extremal construction below shows that the bound in Theorem \ref{thm:C8C10} is sharp.
\begin{figure}[H]
    \centering
\includegraphics{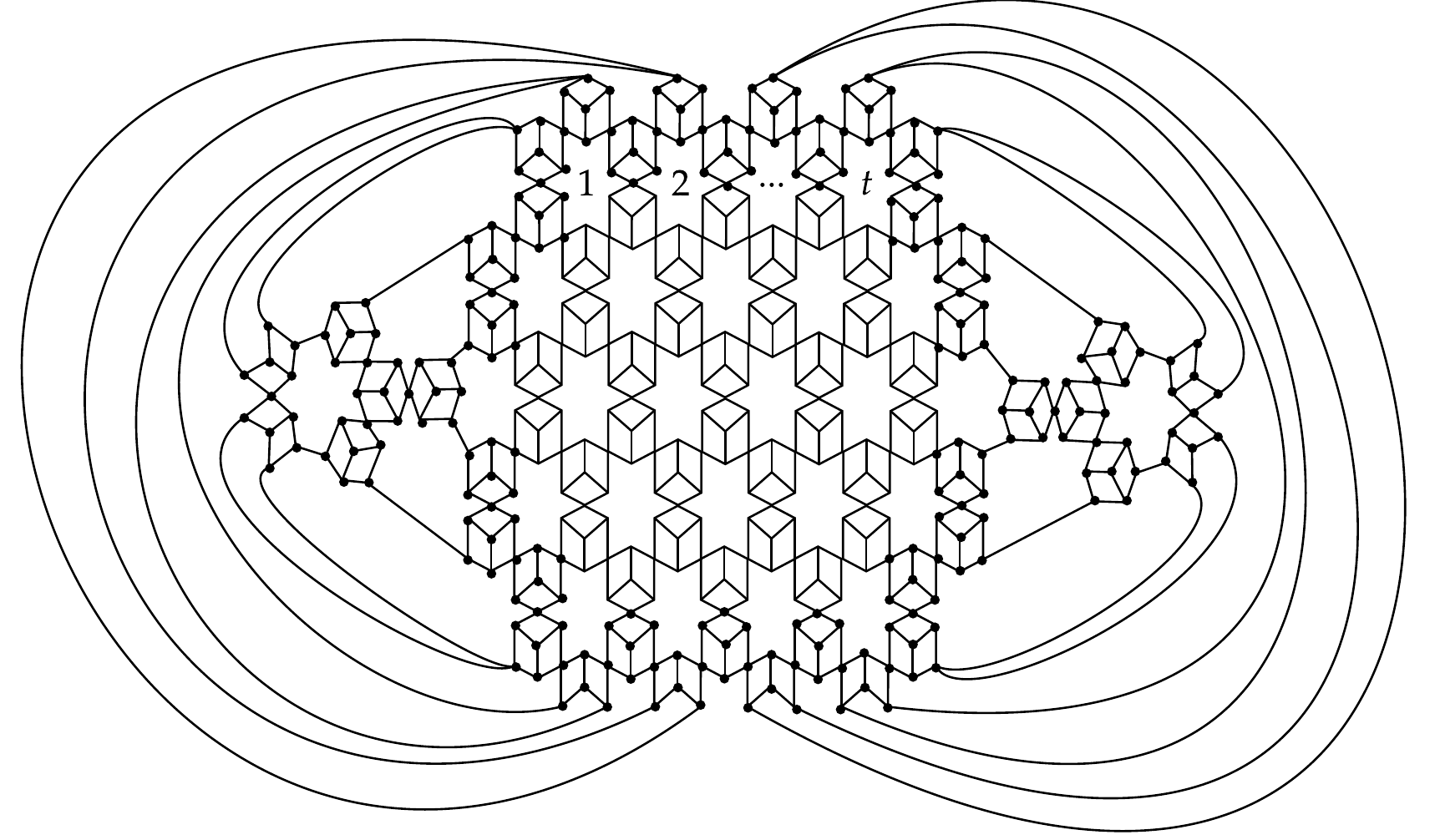}
    \caption{The extremal construction of Theorem \ref{thm:C8C10}.}
    \label{fig:C8C10Extremal}
\end{figure}

In Figure \ref{fig:C8C10Extremal}, there are $66t+155$ vertices and $108t+246$ edges, where $t$ is any positive integer, which satisfy the bound $e_G=18v_G/11 - 84/11$. 

\begin{remark}
This construction comes from the proof of Lemma \ref{lem:C8C10}. In this construction, we only used quadrangular blocks $K_2$ and $Q_7$, and each one of them has $0$ ``contribution".
\end{remark}

\subsubsection{Preparatory propositions}~

For any $C_8,C_{10}$-free bipartite plane graph $G$ with $\delta(G)\ge 3$, since $G$ is a $C_8$-free bipartite plane graph with $\delta(G)\ge3$, by Proposition \ref{pro:biC8Block}, we know that a quadrangular block in $G$ contains at most $7$ vertices. Also, it is easy to check that all four quadrangular blocks in Figure \ref{fig:quadraBlocksBiCEight} are valid quadrangular blocks for $G$, hence they are the quadrangular blocks for $G$.

%all the possible quadrangular blocks in $G$ are still the four quadrangular blocks in Figure \ref{fig:quadraBlocksBiCEight}.

\begin{proposition}
\label{pro:C8C10K2Face}
Let $G$ be a planar bipartite graph on $n$ vertices which does not contain $C_8$ or $C_{10}$ and let $\delta(G)\ge 3$. If there exists a $6$-face in $G$, we can always add another edge to $G$ to get $G'$ such that $G'$ is a planar bipartite graph on $n$ vertices which does not contain $C_8$ or $C_{10}$ and $\delta(G')\ge 3$.
\end{proposition}
\begin{proof}
Suppose that there is a face $f$ of length $6$ in $G$, any face adjacent to $f$ must have length at least $12$, since $\delta(G)\geq 3$ and $C_4$, $C_6$ sharing an edge with $C_6$ creates a $C_8$, $C_{10}$, respectively, and $G$ is $C_8, C_{10}$-free. Therefore, we can safely add a chord in $f$ to make $f$ into a quadrangular block $\Theta_6$ without creating any $C_8$ or $C_{10}$ to get $G'$.
\end{proof}

By Proposition \ref{pro:C8C10K2Face}, it suffices to prove the theorem for graphs which do not have $6$-faces. As for graphs which do have $6$-faces, we can keep adding extra edges to it until there is no $6$-faces and the proof for the resulting graph yields the proof for the original graph.

\subsubsection{Proof of Theorem \ref{thm:C8C10}}~

Now our goal is to prove the following lemma.
\begin{lemma}
\label{lem:C8C10}
Let $G$ be a plane bipartite graph on $n$ vertices which does not contain $C_8$ or $C_{10}$ and let $\delta(G)\ge 3$. Assume that $G$ does not have $6$-faces. Any quadrangular block $B$ in $G$ satisfies 
$$
    24v(B) - 31e(B)+42f(B) \leq 0.
$$
\end{lemma}
Once we have proved this, we have 
$$
    24v_G -31e_G+42f_G = \sum_{B \in \mathcal{Q}}(24v(B)-31e(B)+42f(B)) \leq 0. 
$$ Combining this with Euler's formula $v_G-e_G+f_G = 2$ finishes the proof.
\begin{proof}
For any quadrangular block $B$, each exterior face of $B$ cannot have length $4$ by definition of quadrangular blocks, and it cannot have length $6,8,10$ by assumption. So the length of each exterior face is at least $12$, hence for non-trivial quadrangular blocks, the contribution of an exterior edge to faces is at most $1/12$. Now we do casework to proceed with the proof.

\textbf{Case 1}: $B$ is $K_2$.

Since $\delta(G) \geq 3$, $B$ has $2$ junction vertices, so $v(B) \leq 1/2+1/2$. For face contribution we have $f(B) \leq 1/12+1/12$. With $e(B) = 1$, we have $24v(B)-31e(B)+42f(B) \leq 0$.

\textbf{Case 2}: $B$ is $C_4$.

Since $\delta(G) \geq 3$, all four vertices of $B$ are junction vertices, so $v(B) \leq 4/2 = 2$. There is an interior face and four exterior edges of $B$, so we have  
$
f(B) \leq 1+4/12 = 4/3$. With $e(B) =4$, we obtain $24v(B)-31e(B)+42f(B) \leq -20$.

\textbf{Case 3}: $B$ is $\Theta_6$.

Since $\delta(G) \geq 3$, we have at least $4$ junction vertices, so $v(B) \leq 2+4/2=4$, $e(B)=7$. There are two interior faces and six exterior edges of $B$, so $f(B) \leq 2 +6/12 = 5/2$. With $e(B)=7$ we have $24v(B)-31e(B)+42f(B) \leq -16$.

\textbf{Case 4}: $B$  is $Q_7$.

There are at least $3$ junction vertices, so $v(B) \leq 4+3/2=11/2$. With $e(B)=9$ and $f\leq 3+6/12 = 7/2$, we have $24v(B)-31e(B)+42f(B) \leq 0$.
\end{proof}

\section{Quadrangular blocks in triangle-free graphs}

All graphs considered in this section are triangular-free graphs.

\subsection{\texorpdfstring{$C_6$}{C6}-free graphs}~

In this section, we prove 

\thmTriCSix*

\subsubsection{Extremal construction showing sharpness}~

The extremal construction below shows that the bound in Theorem \ref{thm:triC6} is sharp.
\begin{figure}[H]
    \centering
    \includegraphics[scale = 1]{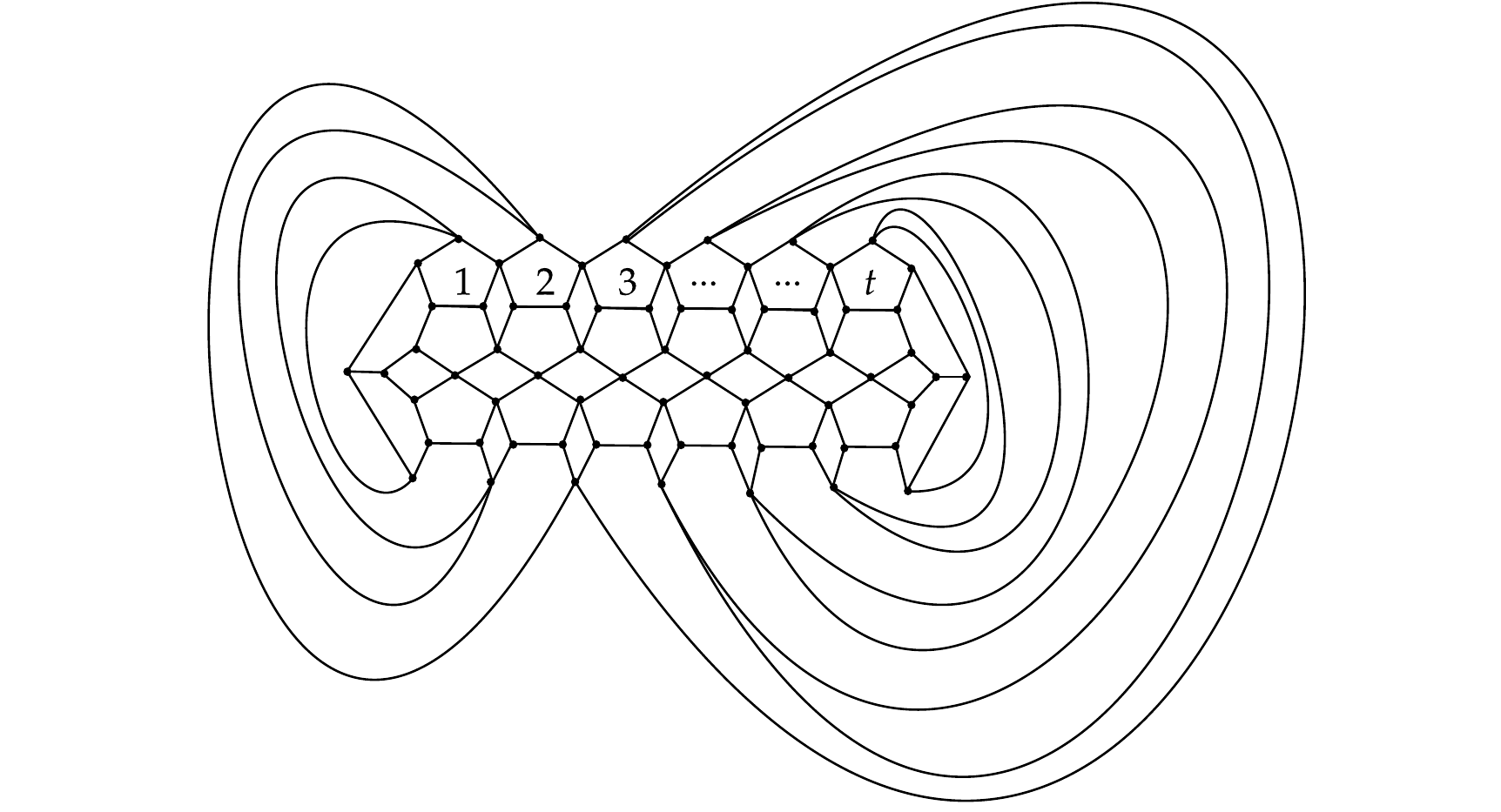}
    \caption{The extremal construction of Theorem \ref{thm:triC6}}
    \label{fig:triC6Extremal}
\end{figure}

In Figure \ref{fig:triC6Extremal}, there are $10t+8$ vertices and $18t+10$ edges where $t$ is any positive integer, which satisfy that $e_G=\lfloor 9v_G/5 - 4 \rfloor$.

\begin{remark}
This construction comes from the proof of Lemma \ref{lem:triC6}. In the extremal construction, other than eight trivial quadrangular blocks $K_2$, each quadrangular block has $0$ ``contribution".
\end{remark}

\subsubsection{Preparatory propositions}

\begin{proposition}
    A quadrangular block in a $C_6$-free triangle-free plane graph $G$ with $\delta(G)\ge 3$ contains at most $4$ vertices.
\end{proposition}

\begin{proof}
    For a quadrangular block which contains more than $4$ vertices, by the generating algorithm of a quadrangular block, it must be initiated from two adjacent $4$-faces and they share only one edge as $\delta(G)\ge 3$, which gives a $C_6$, hence a contradiction.
\end{proof}

Now we describe all the possible quadrangular blocks in $G$. As shown in Figure \ref{fig:triC6}, other than the trivial quadrangular block $K_2$, there is only one quadrangular block on $4$ vertices, which is $C_4$.
\begin{figure}[H]
    \centering
    \begin{tikzpicture}[x=0.75pt,y=0.75pt,yscale=-1,xscale=1]
%uncomment if require: \path (0,300); %set diagram left start at 0, and has height of 300

%Straight Lines [id:da12449728351581624] 
\draw [line width=2.25]    (67.82,183.5) -- (185,183.5) ;
%Shape: Circle [id:dp9702811535191691] 
\draw  [fill={rgb, 255:red, 0; green, 0; blue, 0 }  ,fill opacity=1 ] (61.82,183.5) .. controls (61.82,180.19) and (64.51,177.5) .. (67.82,177.5) .. controls (71.14,177.5) and (73.82,180.19) .. (73.82,183.5) .. controls (73.82,186.81) and (71.14,189.5) .. (67.82,189.5) .. controls (64.51,189.5) and (61.82,186.81) .. (61.82,183.5) -- cycle ;
%Shape: Circle [id:dp41256648961066356] 
\draw  [fill={rgb, 255:red, 0; green, 0; blue, 0 }  ,fill opacity=1 ] (179,183.5) .. controls (179,180.19) and (181.69,177.5) .. (185,177.5) .. controls (188.31,177.5) and (191,180.19) .. (191,183.5) .. controls (191,186.81) and (188.31,189.5) .. (185,189.5) .. controls (181.69,189.5) and (179,186.81) .. (179,183.5) -- cycle ;
%Shape: Square [id:dp12211937398173656] 
\draw  [line width=2.25]  (297.91,124.91) -- (356.5,183.5) -- (297.91,242.09) -- (239.32,183.5) -- cycle ;
%Shape: Circle [id:dp3981287568844698] 
\draw  [fill={rgb, 255:red, 0; green, 0; blue, 0 }  ,fill opacity=1 ] (233.32,183.5) .. controls (233.32,180.19) and (236.01,177.5) .. (239.32,177.5) .. controls (242.64,177.5) and (245.32,180.19) .. (245.32,183.5) .. controls (245.32,186.81) and (242.64,189.5) .. (239.32,189.5) .. controls (236.01,189.5) and (233.32,186.81) .. (233.32,183.5) -- cycle ;
%Shape: Circle [id:dp1582186604678868] 
\draw  [fill={rgb, 255:red, 0; green, 0; blue, 0 }  ,fill opacity=1 ] (350.5,183.5) .. controls (350.5,180.19) and (353.19,177.5) .. (356.5,177.5) .. controls (359.81,177.5) and (362.5,180.19) .. (362.5,183.5) .. controls (362.5,186.81) and (359.81,189.5) .. (356.5,189.5) .. controls (353.19,189.5) and (350.5,186.81) .. (350.5,183.5) -- cycle ;
%Shape: Circle [id:dp47698202330792805] 
\draw  [fill={rgb, 255:red, 0; green, 0; blue, 0 }  ,fill opacity=1 ] (291.91,242.09) .. controls (291.91,238.77) and (294.6,236.09) .. (297.91,236.09) .. controls (301.23,236.09) and (303.91,238.77) .. (303.91,242.09) .. controls (303.91,245.4) and (301.23,248.09) .. (297.91,248.09) .. controls (294.6,248.09) and (291.91,245.4) .. (291.91,242.09) -- cycle ;
%Shape: Circle [id:dp10890630329068807] 
\draw  [fill={rgb, 255:red, 0; green, 0; blue, 0 }  ,fill opacity=1 ] (291.91,124.91) .. controls (291.91,121.6) and (294.6,118.91) .. (297.91,118.91) .. controls (301.23,118.91) and (303.91,121.6) .. (303.91,124.91) .. controls (303.91,128.23) and (301.23,130.91) .. (297.91,130.91) .. controls (294.6,130.91) and (291.91,128.23) .. (291.91,124.91) -- cycle ;

% Text Node
\draw (112,258) node [anchor=north west][inner sep=0.75pt]  [font=\Large] [align=left] {$K_2$};
% Text Node
\draw (287,258) node [anchor=north west][inner sep=0.75pt]  [font=\Large] [align=left] {$C_4$};
\end{tikzpicture}
    \caption{All the possible quadrangular blocks in $G$}
    \label{fig:triC6}
\end{figure}
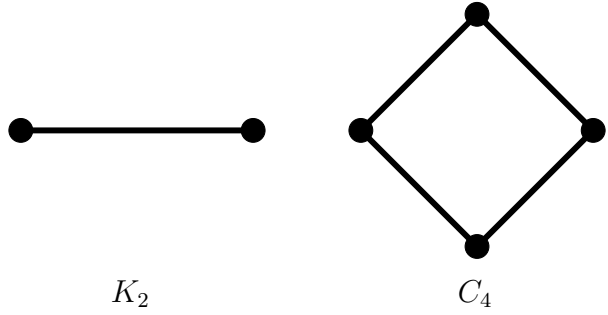

The following proposition is useful in the proof.
\begin{proposition}
\label{pro:triC6}
Let $G$ be a $C_6$-free triangle-free plane graph on $n$ vertices with $\delta(G)\ge3$, then its exterior faces have length at least $5$.
\end{proposition}

\begin{proof}
Since $G$ is triangle-free, its exterior faces cannot have length $3$. By definition of quadrangular blocks, exterior faces cannot have length $4$. So each exterior face has length at least $5$.
\end{proof}

\subsubsection{Proof of Theorem \ref{thm:triC6}}~

Our goal is to show that
\begin{lemma}
\label{lem:triC6}
Let $G$ be a $C_6$-free triangle-free plane graph on $n$ vertices with $\delta(G)\ge 3$. Any quadrangular block $B$ in $G$ satisfies $$
v(B)-5e(B)+10f(B)\le 0.
$$
\end{lemma}

Once we have proved this, we have $$
v_G-5e_G+10f_G = \sum_{B\in\mathcal{Q}}(v(B)-5e(B)+10f(B)) \le 0.
$$ Combining this with Euler's formula $v_G-e_G+f_G=2$ gives $e_G\le 9n/5-4$. Note that $e_G$ is an integer, the proof is finished.

\begin{proof}
We do casework to proceed with the proof. We distinguish the cases according to $B$.

\textbf{Case1}: $B$ is $K_2$.

Since $\delta(G)\ge 3$ we know that both vertices are junction vertices, so $v(B)\le 1/2+1/2$. By Proposition \ref{pro:triC6} we know that $f(B)\le 1/5+1/5$. With $e_G=1$ we have $v(B)-5e(B)+10f(B)\le 0$.

\textbf{Case2}: $B$ is $C_4$.

Since $\delta(G)\ge3$ we know that all of the vertices in $B$ are junction vertices, so $v(B)\le 4\times 1/2$. By Proposition \ref{pro:triC6} we know that $f(B)\le 1+ 4/5$. With $e_G=4$ we have $v(B)-5e(B)+10f(B)\le 0$.
\end{proof}

\subsection{\texorpdfstring{$C_8$}{C8}-free graphs}~

In this section, we prove 
\thmTriCEight*

\subsubsection{A construction showing a lower bound of maximum number of edges}~
 
 The construction below shows that for infinitely many $n$, there exists a $C_8$-free triangle-free planar graph on $n$ vertices with $$
    e_G\geq \frac{9}{5}n-\frac{38}{5}.
    $$
    \begin{figure}[H]
        \centering
        \begin{tikzpicture}[x=0.75pt,y=0.75pt,yscale=-1.8,xscale=1.8]
%uncomment if require: \path (0,165); %set diagram left start at 0, and has height of 165

%Shape: Grid [id:dp9132740345810644] 
\draw  [draw opacity=0] (192,4) -- (262,4) -- (262,44) -- (192,44) -- cycle ; \draw   (192,4) -- (192,44)(212,4) -- (212,44)(232,4) -- (232,44)(252,4) -- (252,44) ; \draw   (192,4) -- (262,4)(192,24) -- (262,24) ; \draw    ;
%Shape: Grid [id:dp3433522452305191] 
\draw  [draw opacity=0] (192,42) -- (262,42) -- (262,82) -- (192,82) -- cycle ; \draw   (192,42) -- (192,82)(212,42) -- (212,82)(232,42) -- (232,82)(252,42) -- (252,82) ; \draw   (192,42) -- (262,42)(192,62) -- (262,62) ; \draw    ;
%Shape: Grid [id:dp574018022444265] 
\draw  [draw opacity=0] (252.2,4) -- (322.2,4) -- (322.2,44) -- (252.2,44) -- cycle ; \draw   (252.2,4) -- (252.2,44)(272.2,4) -- (272.2,44)(292.2,4) -- (292.2,44)(312.2,4) -- (312.2,44) ; \draw   (252.2,4) -- (322.2,4)(252.2,24) -- (322.2,24) ; \draw    ;
%Shape: Grid [id:dp37724323827631734] 
\draw  [draw opacity=0] (252.2,42) -- (322.2,42) -- (322.2,82) -- (252.2,82) -- cycle ; \draw   (252.2,42) -- (252.2,82)(272.2,42) -- (272.2,82)(292.2,42) -- (292.2,82)(312.2,42) -- (312.2,82) ; \draw   (252.2,42) -- (322.2,42)(252.2,62) -- (322.2,62) ; \draw    ;
%Shape: Grid [id:dp24793443693672845] 
\draw  [draw opacity=0] (312.4,4.2) -- (382.4,4.2) -- (382.4,44.2) -- (312.4,44.2) -- cycle ; \draw   (312.4,4.2) -- (312.4,44.2)(332.4,4.2) -- (332.4,44.2)(352.4,4.2) -- (352.4,44.2)(372.4,4.2) -- (372.4,44.2) ; \draw   (312.4,4.2) -- (382.4,4.2)(312.4,24.2) -- (382.4,24.2) ; \draw    ;
%Shape: Grid [id:dp8645214451149326] 
\draw  [draw opacity=0] (312.4,42.2) -- (382.4,42.2) -- (382.4,82.2) -- (312.4,82.2) -- cycle ; \draw   (312.4,42.2) -- (312.4,82.2)(332.4,42.2) -- (332.4,82.2)(352.4,42.2) -- (352.4,82.2)(372.4,42.2) -- (372.4,82.2) ; \draw   (312.4,42.2) -- (382.4,42.2)(312.4,62.2) -- (382.4,62.2) ; \draw    ;
%Shape: Grid [id:dp06259309602238505] 
\draw  [draw opacity=0] (372.6,4.2) -- (442.6,4.2) -- (442.6,44.2) -- (372.6,44.2) -- cycle ; \draw   (372.6,4.2) -- (372.6,44.2)(392.6,4.2) -- (392.6,44.2)(412.6,4.2) -- (412.6,44.2)(432.6,4.2) -- (432.6,44.2) ; \draw   (372.6,4.2) -- (442.6,4.2)(372.6,24.2) -- (442.6,24.2) ; \draw    ;
%Shape: Grid [id:dp600649634643708] 
\draw  [draw opacity=0] (372.6,42.2) -- (442.6,42.2) -- (442.6,82.2) -- (372.6,82.2) -- cycle ; \draw   (372.6,42.2) -- (372.6,82.2)(392.6,42.2) -- (392.6,82.2)(412.6,42.2) -- (412.6,82.2)(432.6,42.2) -- (432.6,82.2) ; \draw   (372.6,42.2) -- (442.6,42.2)(372.6,62.2) -- (442.6,62.2) ; \draw    ;
%Shape: Circle [id:dp6244514091494979] 
\draw  [fill={rgb, 255:red, 0; green, 0; blue, 0 }  ,fill opacity=1 ] (474.8,74.9) .. controls (474.8,73.85) and (475.65,73) .. (476.7,73) .. controls (477.75,73) and (478.6,73.85) .. (478.6,74.9) .. controls (478.6,75.95) and (477.75,76.8) .. (476.7,76.8) .. controls (475.65,76.8) and (474.8,75.95) .. (474.8,74.9) -- cycle ;
%Shape: Circle [id:dp4254704019758966] 
\draw  [fill={rgb, 255:red, 0; green, 0; blue, 0 }  ,fill opacity=1 ] (189.4,14.1) .. controls (189.4,13.05) and (190.25,12.2) .. (191.3,12.2) .. controls (192.35,12.2) and (193.2,13.05) .. (193.2,14.1) .. controls (193.2,15.15) and (192.35,16) .. (191.3,16) .. controls (190.25,16) and (189.4,15.15) .. (189.4,14.1) -- cycle ;
%Shape: Circle [id:dp05106747064033712] 
\draw  [fill={rgb, 255:red, 0; green, 0; blue, 0 }  ,fill opacity=1 ] (209.4,34.1) .. controls (209.4,33.05) and (210.25,32.2) .. (211.3,32.2) .. controls (212.35,32.2) and (213.2,33.05) .. (213.2,34.1) .. controls (213.2,35.15) and (212.35,36) .. (211.3,36) .. controls (210.25,36) and (209.4,35.15) .. (209.4,34.1) -- cycle ;
%Shape: Circle [id:dp633280487687403] 
\draw  [fill={rgb, 255:red, 0; green, 0; blue, 0 }  ,fill opacity=1 ] (210.4,53.1) .. controls (210.4,52.05) and (211.25,51.2) .. (212.3,51.2) .. controls (213.35,51.2) and (214.2,52.05) .. (214.2,53.1) .. controls (214.2,54.15) and (213.35,55) .. (212.3,55) .. controls (211.25,55) and (210.4,54.15) .. (210.4,53.1) -- cycle ;
%Shape: Circle [id:dp20328826970682368] 
\draw  [fill={rgb, 255:red, 0; green, 0; blue, 0 }  ,fill opacity=1 ] (200.4,62.1) .. controls (200.4,61.05) and (201.25,60.2) .. (202.3,60.2) .. controls (203.35,60.2) and (204.2,61.05) .. (204.2,62.1) .. controls (204.2,63.15) and (203.35,64) .. (202.3,64) .. controls (201.25,64) and (200.4,63.15) .. (200.4,62.1) -- cycle ;
%Shape: Circle [id:dp9784563092549619] 
\draw  [fill={rgb, 255:red, 0; green, 0; blue, 0 }  ,fill opacity=1 ] (190.4,72.1) .. controls (190.4,71.05) and (191.25,70.2) .. (192.3,70.2) .. controls (193.35,70.2) and (194.2,71.05) .. (194.2,72.1) .. controls (194.2,73.15) and (193.35,74) .. (192.3,74) .. controls (191.25,74) and (190.4,73.15) .. (190.4,72.1) -- cycle ;
%Shape: Circle [id:dp8286581162032636] 
\draw  [fill={rgb, 255:red, 0; green, 0; blue, 0 }  ,fill opacity=1 ] (220.4,23.1) .. controls (220.4,22.05) and (221.25,21.2) .. (222.3,21.2) .. controls (223.35,21.2) and (224.2,22.05) .. (224.2,23.1) .. controls (224.2,24.15) and (223.35,25) .. (222.3,25) .. controls (221.25,25) and (220.4,24.15) .. (220.4,23.1) -- cycle ;
%Shape: Circle [id:dp8104925011913635] 
\draw  [fill={rgb, 255:red, 0; green, 0; blue, 0 }  ,fill opacity=1 ] (220.4,4.1) .. controls (220.4,3.05) and (221.25,2.2) .. (222.3,2.2) .. controls (223.35,2.2) and (224.2,3.05) .. (224.2,4.1) .. controls (224.2,5.15) and (223.35,6) .. (222.3,6) .. controls (221.25,6) and (220.4,5.15) .. (220.4,4.1) -- cycle ;
%Shape: Circle [id:dp9093424023962735] 
\draw  [fill={rgb, 255:red, 0; green, 0; blue, 0 }  ,fill opacity=1 ] (240.4,24.1) .. controls (240.4,23.05) and (241.25,22.2) .. (242.3,22.2) .. controls (243.35,22.2) and (244.2,23.05) .. (244.2,24.1) .. controls (244.2,25.15) and (243.35,26) .. (242.3,26) .. controls (241.25,26) and (240.4,25.15) .. (240.4,24.1) -- cycle ;
%Shape: Circle [id:dp0075968682195617365] 
\draw  [fill={rgb, 255:red, 0; green, 0; blue, 0 }  ,fill opacity=1 ] (260.4,42.1) .. controls (260.4,41.05) and (261.25,40.2) .. (262.3,40.2) .. controls (263.35,40.2) and (264.2,41.05) .. (264.2,42.1) .. controls (264.2,43.15) and (263.35,44) .. (262.3,44) .. controls (261.25,44) and (260.4,43.15) .. (260.4,42.1) -- cycle ;
%Shape: Circle [id:dp3724902525437297] 
\draw  [fill={rgb, 255:red, 0; green, 0; blue, 0 }  ,fill opacity=1 ] (230.6,72.3) .. controls (230.6,71.25) and (231.45,70.4) .. (232.5,70.4) .. controls (233.55,70.4) and (234.4,71.25) .. (234.4,72.3) .. controls (234.4,73.35) and (233.55,74.2) .. (232.5,74.2) .. controls (231.45,74.2) and (230.6,73.35) .. (230.6,72.3) -- cycle ;
%Shape: Circle [id:dp9302588188176433] 
\draw  [fill={rgb, 255:red, 0; green, 0; blue, 0 }  ,fill opacity=1 ] (270.4,71.7) .. controls (270.4,70.65) and (271.25,69.8) .. (272.3,69.8) .. controls (273.35,69.8) and (274.2,70.65) .. (274.2,71.7) .. controls (274.2,72.75) and (273.35,73.6) .. (272.3,73.6) .. controls (271.25,73.6) and (270.4,72.75) .. (270.4,71.7) -- cycle ;
%Shape: Circle [id:dp23779948050077704] 
\draw  [fill={rgb, 255:red, 0; green, 0; blue, 0 }  ,fill opacity=1 ] (261.2,62.5) .. controls (261.2,61.45) and (262.05,60.6) .. (263.1,60.6) .. controls (264.15,60.6) and (265,61.45) .. (265,62.5) .. controls (265,63.55) and (264.15,64.4) .. (263.1,64.4) .. controls (262.05,64.4) and (261.2,63.55) .. (261.2,62.5) -- cycle ;
%Shape: Circle [id:dp6971612427084728] 
\draw  [fill={rgb, 255:red, 0; green, 0; blue, 0 }  ,fill opacity=1 ] (250.2,53.3) .. controls (250.2,52.25) and (251.05,51.4) .. (252.1,51.4) .. controls (253.15,51.4) and (254,52.25) .. (254,53.3) .. controls (254,54.35) and (253.15,55.2) .. (252.1,55.2) .. controls (251.05,55.2) and (250.2,54.35) .. (250.2,53.3) -- cycle ;
%Shape: Circle [id:dp8375323232931333] 
\draw  [fill={rgb, 255:red, 0; green, 0; blue, 0 }  ,fill opacity=1 ] (270.2,14.3) .. controls (270.2,13.25) and (271.05,12.4) .. (272.1,12.4) .. controls (273.15,12.4) and (274,13.25) .. (274,14.3) .. controls (274,15.35) and (273.15,16.2) .. (272.1,16.2) .. controls (271.05,16.2) and (270.2,15.35) .. (270.2,14.3) -- cycle ;
%Shape: Circle [id:dp4992679486778717] 
\draw  [fill={rgb, 255:red, 0; green, 0; blue, 0 }  ,fill opacity=1 ] (230.4,14.1) .. controls (230.4,13.05) and (231.25,12.2) .. (232.3,12.2) .. controls (233.35,12.2) and (234.2,13.05) .. (234.2,14.1) .. controls (234.2,15.15) and (233.35,16) .. (232.3,16) .. controls (231.25,16) and (230.4,15.15) .. (230.4,14.1) -- cycle ;
%Shape: Circle [id:dp48963699716853104] 
\draw  [fill={rgb, 255:red, 0; green, 0; blue, 0 }  ,fill opacity=1 ] (200.2,3.3) .. controls (200.2,2.25) and (201.05,1.4) .. (202.1,1.4) .. controls (203.15,1.4) and (204,2.25) .. (204,3.3) .. controls (204,4.35) and (203.15,5.2) .. (202.1,5.2) .. controls (201.05,5.2) and (200.2,4.35) .. (200.2,3.3) -- cycle ;
%Shape: Circle [id:dp62060375341251] 
\draw  [fill={rgb, 255:red, 0; green, 0; blue, 0 }  ,fill opacity=1 ] (240.4,41.1) .. controls (240.4,40.05) and (241.25,39.2) .. (242.3,39.2) .. controls (243.35,39.2) and (244.2,40.05) .. (244.2,41.1) .. controls (244.2,42.15) and (243.35,43) .. (242.3,43) .. controls (241.25,43) and (240.4,42.15) .. (240.4,41.1) -- cycle ;
%Shape: Circle [id:dp8037543512283716] 
\draw  [fill={rgb, 255:red, 0; green, 0; blue, 0 }  ,fill opacity=1 ] (250.4,34.1) .. controls (250.4,33.05) and (251.25,32.2) .. (252.3,32.2) .. controls (253.35,32.2) and (254.2,33.05) .. (254.2,34.1) .. controls (254.2,35.15) and (253.35,36) .. (252.3,36) .. controls (251.25,36) and (250.4,35.15) .. (250.4,34.1) -- cycle ;
%Shape: Circle [id:dp9208155538208562] 
\draw  [fill={rgb, 255:red, 0; green, 0; blue, 0 }  ,fill opacity=1 ] (465.6,75.1) .. controls (465.6,74.05) and (466.45,73.2) .. (467.5,73.2) .. controls (468.55,73.2) and (469.4,74.05) .. (469.4,75.1) .. controls (469.4,76.15) and (468.55,77) .. (467.5,77) .. controls (466.45,77) and (465.6,76.15) .. (465.6,75.1) -- cycle ;
%Shape: Circle [id:dp018296377193624425] 
\draw  [fill={rgb, 255:red, 0; green, 0; blue, 0 }  ,fill opacity=1 ] (484.8,74.9) .. controls (484.8,73.85) and (485.65,73) .. (486.7,73) .. controls (487.75,73) and (488.6,73.85) .. (488.6,74.9) .. controls (488.6,75.95) and (487.75,76.8) .. (486.7,76.8) .. controls (485.65,76.8) and (484.8,75.95) .. (484.8,74.9) -- cycle ;
%Shape: Circle [id:dp5695182675809054] 
\draw  [fill={rgb, 255:red, 0; green, 0; blue, 0 }  ,fill opacity=1 ] (200.2,160.5) .. controls (200.2,159.45) and (201.05,158.6) .. (202.1,158.6) .. controls (203.15,158.6) and (204,159.45) .. (204,160.5) .. controls (204,161.55) and (203.15,162.4) .. (202.1,162.4) .. controls (201.05,162.4) and (200.2,161.55) .. (200.2,160.5) -- cycle ;
%Shape: Circle [id:dp4886714110566286] 
\draw  [fill={rgb, 255:red, 0; green, 0; blue, 0 }  ,fill opacity=1 ] (220.2,160.5) .. controls (220.2,159.45) and (221.05,158.6) .. (222.1,158.6) .. controls (223.15,158.6) and (224,159.45) .. (224,160.5) .. controls (224,161.55) and (223.15,162.4) .. (222.1,162.4) .. controls (221.05,162.4) and (220.2,161.55) .. (220.2,160.5) -- cycle ;
%Shape: Circle [id:dp8874366357869778] 
\draw  [fill={rgb, 255:red, 0; green, 0; blue, 0 }  ,fill opacity=1 ] (220.2,82.5) .. controls (220.2,81.45) and (221.05,80.6) .. (222.1,80.6) .. controls (223.15,80.6) and (224,81.45) .. (224,82.5) .. controls (224,83.55) and (223.15,84.4) .. (222.1,84.4) .. controls (221.05,84.4) and (220.2,83.55) .. (220.2,82.5) -- cycle ;
%Shape: Circle [id:dp878325163821899] 
\draw  [fill={rgb, 255:red, 0; green, 0; blue, 0 }  ,fill opacity=1 ] (200.2,82.5) .. controls (200.2,81.45) and (201.05,80.6) .. (202.1,80.6) .. controls (203.15,80.6) and (204,81.45) .. (204,82.5) .. controls (204,83.55) and (203.15,84.4) .. (202.1,84.4) .. controls (201.05,84.4) and (200.2,83.55) .. (200.2,82.5) -- cycle ;
%Straight Lines [id:da7281639248214185] 
\draw    (191.37,160.68) -- (441.87,160.77) ;
%Curve Lines [id:da17237048901969887] 
\draw    (192.1,150.5) .. controls (150.6,156.3) and (157.6,58.3) .. (191.9,72.1) ;
%Straight Lines [id:da3015439340255952] 
\draw    (202.1,5.2) -- (202.6,14.41) ;
%Straight Lines [id:da7336432760379838] 
\draw    (191.3,14.1) -- (202.6,14.41) ;
%Straight Lines [id:da9460123799732778] 
\draw    (202.6,14.41) -- (212,24) ;
%Shape: Circle [id:dp2963442812338817] 
\draw  [fill={rgb, 255:red, 0; green, 0; blue, 0 }  ,fill opacity=1 ] (200.2,14.3) .. controls (200.2,13.25) and (201.05,12.4) .. (202.1,12.4) .. controls (203.15,12.4) and (204,13.25) .. (204,14.3) .. controls (204,15.35) and (203.15,16.2) .. (202.1,16.2) .. controls (201.05,16.2) and (200.2,15.35) .. (200.2,14.3) -- cycle ;
%Straight Lines [id:da6995463422022321] 
\draw    (202.8,61.2) -- (202.3,51.99) ;
%Straight Lines [id:da7318123893641288] 
\draw    (213.6,52.3) -- (202.3,51.99) ;
%Straight Lines [id:da9201072322927597] 
\draw    (202.3,51.99) -- (192.9,42.4) ;
%Shape: Circle [id:dp30174101792162245] 
\draw  [fill={rgb, 255:red, 0; green, 0; blue, 0 }  ,fill opacity=1 ] (204.7,52.1) .. controls (204.7,53.15) and (203.85,54) .. (202.8,54) .. controls (201.75,54) and (200.9,53.15) .. (200.9,52.1) .. controls (200.9,51.05) and (201.75,50.2) .. (202.8,50.2) .. controls (203.85,50.2) and (204.7,51.05) .. (204.7,52.1) -- cycle ;
%Straight Lines [id:da435525153669027] 
\draw    (222.9,23.8) -- (223.4,33.01) ;
%Straight Lines [id:da13091206725778592] 
\draw    (212.1,32.7) -- (223.4,33.01) ;
%Straight Lines [id:da9567378433730878] 
\draw    (223.4,33.01) -- (232.8,42.6) ;
%Shape: Circle [id:dp8184380554211255] 
\draw  [fill={rgb, 255:red, 0; green, 0; blue, 0 }  ,fill opacity=1 ] (221,32.9) .. controls (221,31.85) and (221.85,31) .. (222.9,31) .. controls (223.95,31) and (224.8,31.85) .. (224.8,32.9) .. controls (224.8,33.95) and (223.95,34.8) .. (222.9,34.8) .. controls (221.85,34.8) and (221,33.95) .. (221,32.9) -- cycle ;
%Straight Lines [id:da7711002677955052] 
\draw    (251.05,52.02) -- (241.85,52.55) ;
%Straight Lines [id:da1558622359396502] 
\draw    (242.12,41.25) -- (241.85,52.55) ;
%Straight Lines [id:da8944104940021793] 
\draw    (241.85,52.55) -- (232.28,61.98) ;
%Shape: Circle [id:dp36407954905496975] 
\draw  [fill={rgb, 255:red, 0; green, 0; blue, 0 }  ,fill opacity=1 ] (241.95,50.15) .. controls (242.99,50.15) and (243.85,50.99) .. (243.85,52.04) .. controls (243.85,53.09) and (243.01,53.95) .. (241.96,53.95) .. controls (240.91,53.95) and (240.05,53.1) .. (240.05,52.06) .. controls (240.05,51.01) and (240.9,50.15) .. (241.95,50.15) -- cycle ;
%Straight Lines [id:da467570520236509] 
\draw    (253.09,34.34) -- (262.29,33.88) ;
%Straight Lines [id:da4580997674510261] 
\draw    (262.3,42.1) -- (262.29,33.88) ;
%Straight Lines [id:da2950170382146724] 
\draw    (262.29,33.88) -- (271.93,24.52) ;
%Shape: Circle [id:dp38002853074006193] 
\draw  [fill={rgb, 255:red, 0; green, 0; blue, 0 }  ,fill opacity=1 ] (262.18,36.28) .. controls (261.13,36.27) and (260.28,35.42) .. (260.29,34.37) .. controls (260.29,33.32) and (261.14,32.47) .. (262.19,32.48) .. controls (263.24,32.48) and (264.09,33.34) .. (264.09,34.38) .. controls (264.08,35.43) and (263.23,36.28) .. (262.18,36.28) -- cycle ;
%Straight Lines [id:da616693581949251] 
\draw    (232.69,13.94) -- (241.89,13.48) ;
%Straight Lines [id:da4649413281880943] 
\draw    (241.54,24.78) -- (241.89,13.48) ;
%Straight Lines [id:da4668854015698982] 
\draw    (241.89,13.48) -- (251.53,4.12) ;
%Shape: Circle [id:dp5710949388757123] 
\draw  [fill={rgb, 255:red, 0; green, 0; blue, 0 }  ,fill opacity=1 ] (241.78,15.88) .. controls (240.73,15.87) and (239.88,15.02) .. (239.89,13.97) .. controls (239.89,12.92) and (240.74,12.07) .. (241.79,12.08) .. controls (242.84,12.08) and (243.69,12.94) .. (243.69,13.98) .. controls (243.68,15.03) and (242.83,15.88) .. (241.78,15.88) -- cycle ;
%Straight Lines [id:da3201024179666343] 
\draw    (272.65,71.17) -- (263.45,71.75) ;
%Straight Lines [id:da8933692162150204] 
\draw    (263.67,60.45) -- (263.45,71.75) ;
%Straight Lines [id:da13395185351523198] 
\draw    (263.45,71.75) -- (252.2,82.41) ;
%Shape: Circle [id:dp31432862140132456] 
\draw  [fill={rgb, 255:red, 0; green, 0; blue, 0 }  ,fill opacity=1 ] (263.54,69.35) .. controls (264.59,69.34) and (265.45,70.18) .. (265.45,71.23) .. controls (265.46,72.28) and (264.62,73.14) .. (263.57,73.15) .. controls (262.52,73.16) and (261.66,72.31) .. (261.65,71.26) .. controls (261.65,70.21) and (262.49,69.36) .. (263.54,69.35) -- cycle ;
%Straight Lines [id:da4006791211134919] 
\draw    (222.43,81.79) -- (221.98,72.58) ;
%Straight Lines [id:da3820538926728201] 
\draw    (233.28,72.95) -- (221.98,72.58) ;
%Straight Lines [id:da7528328634307735] 
\draw    (221.98,72.58) -- (212,62) ;
%Shape: Circle [id:dp49587307848234774] 
\draw  [fill={rgb, 255:red, 0; green, 0; blue, 0 }  ,fill opacity=1 ] (224.38,72.7) .. controls (224.37,73.75) and (223.52,74.6) .. (222.47,74.59) .. controls (221.42,74.59) and (220.57,73.73) .. (220.58,72.68) .. controls (220.59,71.63) and (221.44,70.79) .. (222.49,70.79) .. controls (223.54,70.8) and (224.39,71.65) .. (224.38,72.7) -- cycle ;
%Shape: Grid [id:dp6153860258664372] 
\draw  [draw opacity=0] (192,82.4) -- (262,82.4) -- (262,122.4) -- (192,122.4) -- cycle ; \draw   (192,82.4) -- (192,122.4)(212,82.4) -- (212,122.4)(232,82.4) -- (232,122.4)(252,82.4) -- (252,122.4) ; \draw   (192,82.4) -- (262,82.4)(192,102.4) -- (262,102.4) ; \draw    ;
%Shape: Grid [id:dp027352916296093044] 
\draw  [draw opacity=0] (192,120.4) -- (262,120.4) -- (262,160.4) -- (192,160.4) -- cycle ; \draw   (192,120.4) -- (192,160.4)(212,120.4) -- (212,160.4)(232,120.4) -- (232,160.4)(252,120.4) -- (252,160.4) ; \draw   (192,120.4) -- (262,120.4)(192,140.4) -- (262,140.4) ; \draw    ;
%Shape: Grid [id:dp3318875797552985] 
\draw  [draw opacity=0] (252.2,82.4) -- (322.2,82.4) -- (322.2,122.4) -- (252.2,122.4) -- cycle ; \draw   (252.2,82.4) -- (252.2,122.4)(272.2,82.4) -- (272.2,122.4)(292.2,82.4) -- (292.2,122.4)(312.2,82.4) -- (312.2,122.4) ; \draw   (252.2,82.4) -- (322.2,82.4)(252.2,102.4) -- (322.2,102.4) ; \draw    ;
%Shape: Grid [id:dp9294124534860426] 
\draw  [draw opacity=0] (252.2,120.4) -- (322.2,120.4) -- (322.2,160.4) -- (252.2,160.4) -- cycle ; \draw   (252.2,120.4) -- (252.2,160.4)(272.2,120.4) -- (272.2,160.4)(292.2,120.4) -- (292.2,160.4)(312.2,120.4) -- (312.2,160.4) ; \draw   (252.2,120.4) -- (322.2,120.4)(252.2,140.4) -- (322.2,140.4) ; \draw    ;
%Shape: Grid [id:dp03455682926056047] 
\draw  [draw opacity=0] (312.4,82.6) -- (382.4,82.6) -- (382.4,122.6) -- (312.4,122.6) -- cycle ; \draw   (312.4,82.6) -- (312.4,122.6)(332.4,82.6) -- (332.4,122.6)(352.4,82.6) -- (352.4,122.6)(372.4,82.6) -- (372.4,122.6) ; \draw   (312.4,82.6) -- (382.4,82.6)(312.4,102.6) -- (382.4,102.6) ; \draw    ;
%Shape: Grid [id:dp7901468497915178] 
\draw  [draw opacity=0] (312.4,120.6) -- (382.4,120.6) -- (382.4,160.6) -- (312.4,160.6) -- cycle ; \draw   (312.4,120.6) -- (312.4,160.6)(332.4,120.6) -- (332.4,160.6)(352.4,120.6) -- (352.4,160.6)(372.4,120.6) -- (372.4,160.6) ; \draw   (312.4,120.6) -- (382.4,120.6)(312.4,140.6) -- (382.4,140.6) ; \draw    ;
%Shape: Grid [id:dp07579229118864728] 
\draw  [draw opacity=0] (372.6,82.6) -- (442.6,82.6) -- (442.6,122.6) -- (372.6,122.6) -- cycle ; \draw   (372.6,82.6) -- (372.6,122.6)(392.6,82.6) -- (392.6,122.6)(412.6,82.6) -- (412.6,122.6)(432.6,82.6) -- (432.6,122.6) ; \draw   (372.6,82.6) -- (442.6,82.6)(372.6,102.6) -- (442.6,102.6) ; \draw    ;
%Shape: Grid [id:dp20844935462182468] 
\draw  [draw opacity=0] (372.6,120.6) -- (442.6,120.6) -- (442.6,160.6) -- (372.6,160.6) -- cycle ; \draw   (372.6,120.6) -- (372.6,160.6)(392.6,120.6) -- (392.6,160.6)(412.6,120.6) -- (412.6,160.6)(432.6,120.6) -- (432.6,160.6) ; \draw   (372.6,120.6) -- (442.6,120.6)(372.6,140.6) -- (442.6,140.6) ; \draw    ;
%Shape: Circle [id:dp6985731818391674] 
\draw  [fill={rgb, 255:red, 0; green, 0; blue, 0 }  ,fill opacity=1 ] (189.4,92.5) .. controls (189.4,91.45) and (190.25,90.6) .. (191.3,90.6) .. controls (192.35,90.6) and (193.2,91.45) .. (193.2,92.5) .. controls (193.2,93.55) and (192.35,94.4) .. (191.3,94.4) .. controls (190.25,94.4) and (189.4,93.55) .. (189.4,92.5) -- cycle ;
%Shape: Circle [id:dp26531145180785076] 
\draw  [fill={rgb, 255:red, 0; green, 0; blue, 0 }  ,fill opacity=1 ] (209.4,112.5) .. controls (209.4,111.45) and (210.25,110.6) .. (211.3,110.6) .. controls (212.35,110.6) and (213.2,111.45) .. (213.2,112.5) .. controls (213.2,113.55) and (212.35,114.4) .. (211.3,114.4) .. controls (210.25,114.4) and (209.4,113.55) .. (209.4,112.5) -- cycle ;
%Shape: Circle [id:dp3781644785519307] 
\draw  [fill={rgb, 255:red, 0; green, 0; blue, 0 }  ,fill opacity=1 ] (210.4,131.5) .. controls (210.4,130.45) and (211.25,129.6) .. (212.3,129.6) .. controls (213.35,129.6) and (214.2,130.45) .. (214.2,131.5) .. controls (214.2,132.55) and (213.35,133.4) .. (212.3,133.4) .. controls (211.25,133.4) and (210.4,132.55) .. (210.4,131.5) -- cycle ;
%Shape: Circle [id:dp4271175034583885] 
\draw  [fill={rgb, 255:red, 0; green, 0; blue, 0 }  ,fill opacity=1 ] (200.4,140.5) .. controls (200.4,139.45) and (201.25,138.6) .. (202.3,138.6) .. controls (203.35,138.6) and (204.2,139.45) .. (204.2,140.5) .. controls (204.2,141.55) and (203.35,142.4) .. (202.3,142.4) .. controls (201.25,142.4) and (200.4,141.55) .. (200.4,140.5) -- cycle ;
%Shape: Circle [id:dp8019536439462833] 
\draw  [fill={rgb, 255:red, 0; green, 0; blue, 0 }  ,fill opacity=1 ] (190.4,150.5) .. controls (190.4,149.45) and (191.25,148.6) .. (192.3,148.6) .. controls (193.35,148.6) and (194.2,149.45) .. (194.2,150.5) .. controls (194.2,151.55) and (193.35,152.4) .. (192.3,152.4) .. controls (191.25,152.4) and (190.4,151.55) .. (190.4,150.5) -- cycle ;
%Shape: Circle [id:dp6915657599772871] 
\draw  [fill={rgb, 255:red, 0; green, 0; blue, 0 }  ,fill opacity=1 ] (220.4,101.5) .. controls (220.4,100.45) and (221.25,99.6) .. (222.3,99.6) .. controls (223.35,99.6) and (224.2,100.45) .. (224.2,101.5) .. controls (224.2,102.55) and (223.35,103.4) .. (222.3,103.4) .. controls (221.25,103.4) and (220.4,102.55) .. (220.4,101.5) -- cycle ;
%Shape: Circle [id:dp5498921799064951] 
\draw  [fill={rgb, 255:red, 0; green, 0; blue, 0 }  ,fill opacity=1 ] (240.4,102.5) .. controls (240.4,101.45) and (241.25,100.6) .. (242.3,100.6) .. controls (243.35,100.6) and (244.2,101.45) .. (244.2,102.5) .. controls (244.2,103.55) and (243.35,104.4) .. (242.3,104.4) .. controls (241.25,104.4) and (240.4,103.55) .. (240.4,102.5) -- cycle ;
%Shape: Circle [id:dp34996228122498807] 
\draw  [fill={rgb, 255:red, 0; green, 0; blue, 0 }  ,fill opacity=1 ] (260.4,120.5) .. controls (260.4,119.45) and (261.25,118.6) .. (262.3,118.6) .. controls (263.35,118.6) and (264.2,119.45) .. (264.2,120.5) .. controls (264.2,121.55) and (263.35,122.4) .. (262.3,122.4) .. controls (261.25,122.4) and (260.4,121.55) .. (260.4,120.5) -- cycle ;
%Shape: Circle [id:dp7551441870419577] 
\draw  [fill={rgb, 255:red, 0; green, 0; blue, 0 }  ,fill opacity=1 ] (230.6,150.7) .. controls (230.6,149.65) and (231.45,148.8) .. (232.5,148.8) .. controls (233.55,148.8) and (234.4,149.65) .. (234.4,150.7) .. controls (234.4,151.75) and (233.55,152.6) .. (232.5,152.6) .. controls (231.45,152.6) and (230.6,151.75) .. (230.6,150.7) -- cycle ;
%Shape: Circle [id:dp6015203240957774] 
\draw  [fill={rgb, 255:red, 0; green, 0; blue, 0 }  ,fill opacity=1 ] (270.4,150.1) .. controls (270.4,149.05) and (271.25,148.2) .. (272.3,148.2) .. controls (273.35,148.2) and (274.2,149.05) .. (274.2,150.1) .. controls (274.2,151.15) and (273.35,152) .. (272.3,152) .. controls (271.25,152) and (270.4,151.15) .. (270.4,150.1) -- cycle ;
%Shape: Circle [id:dp2914029444606503] 
\draw  [fill={rgb, 255:red, 0; green, 0; blue, 0 }  ,fill opacity=1 ] (261.2,140.9) .. controls (261.2,139.85) and (262.05,139) .. (263.1,139) .. controls (264.15,139) and (265,139.85) .. (265,140.9) .. controls (265,141.95) and (264.15,142.8) .. (263.1,142.8) .. controls (262.05,142.8) and (261.2,141.95) .. (261.2,140.9) -- cycle ;
%Shape: Circle [id:dp650692160226676] 
\draw  [fill={rgb, 255:red, 0; green, 0; blue, 0 }  ,fill opacity=1 ] (250.2,131.7) .. controls (250.2,130.65) and (251.05,129.8) .. (252.1,129.8) .. controls (253.15,129.8) and (254,130.65) .. (254,131.7) .. controls (254,132.75) and (253.15,133.6) .. (252.1,133.6) .. controls (251.05,133.6) and (250.2,132.75) .. (250.2,131.7) -- cycle ;
%Shape: Circle [id:dp8373441967831252] 
\draw  [fill={rgb, 255:red, 0; green, 0; blue, 0 }  ,fill opacity=1 ] (270.2,92.7) .. controls (270.2,91.65) and (271.05,90.8) .. (272.1,90.8) .. controls (273.15,90.8) and (274,91.65) .. (274,92.7) .. controls (274,93.75) and (273.15,94.6) .. (272.1,94.6) .. controls (271.05,94.6) and (270.2,93.75) .. (270.2,92.7) -- cycle ;
%Shape: Circle [id:dp4710109272780787] 
\draw  [fill={rgb, 255:red, 0; green, 0; blue, 0 }  ,fill opacity=1 ] (230.4,92.5) .. controls (230.4,91.45) and (231.25,90.6) .. (232.3,90.6) .. controls (233.35,90.6) and (234.2,91.45) .. (234.2,92.5) .. controls (234.2,93.55) and (233.35,94.4) .. (232.3,94.4) .. controls (231.25,94.4) and (230.4,93.55) .. (230.4,92.5) -- cycle ;
%Shape: Circle [id:dp02610071394657809] 
\draw  [fill={rgb, 255:red, 0; green, 0; blue, 0 }  ,fill opacity=1 ] (240.4,119.5) .. controls (240.4,118.45) and (241.25,117.6) .. (242.3,117.6) .. controls (243.35,117.6) and (244.2,118.45) .. (244.2,119.5) .. controls (244.2,120.55) and (243.35,121.4) .. (242.3,121.4) .. controls (241.25,121.4) and (240.4,120.55) .. (240.4,119.5) -- cycle ;
%Shape: Circle [id:dp12220721014161451] 
\draw  [fill={rgb, 255:red, 0; green, 0; blue, 0 }  ,fill opacity=1 ] (250.4,112.5) .. controls (250.4,111.45) and (251.25,110.6) .. (252.3,110.6) .. controls (253.35,110.6) and (254.2,111.45) .. (254.2,112.5) .. controls (254.2,113.55) and (253.35,114.4) .. (252.3,114.4) .. controls (251.25,114.4) and (250.4,113.55) .. (250.4,112.5) -- cycle ;
%Shape: Circle [id:dp5869881583542342] 
\draw  [fill={rgb, 255:red, 0; green, 0; blue, 0 }  ,fill opacity=1 ] (220.2,160.9) .. controls (220.2,159.85) and (221.05,159) .. (222.1,159) .. controls (223.15,159) and (224,159.85) .. (224,160.9) .. controls (224,161.95) and (223.15,162.8) .. (222.1,162.8) .. controls (221.05,162.8) and (220.2,161.95) .. (220.2,160.9) -- cycle ;
%Shape: Circle [id:dp20467993496549752] 
\draw  [fill={rgb, 255:red, 0; green, 0; blue, 0 }  ,fill opacity=1 ] (200.2,160.9) .. controls (200.2,159.85) and (201.05,159) .. (202.1,159) .. controls (203.15,159) and (204,159.85) .. (204,160.9) .. controls (204,161.95) and (203.15,162.8) .. (202.1,162.8) .. controls (201.05,162.8) and (200.2,161.95) .. (200.2,160.9) -- cycle ;
%Straight Lines [id:da31265127014288363] 
\draw    (202.1,83.6) -- (202.6,92.81) ;
%Straight Lines [id:da4843737034022304] 
\draw    (191.3,92.5) -- (202.6,92.81) ;
%Straight Lines [id:da284758971574554] 
\draw    (202.6,92.81) -- (212,102.4) ;
%Shape: Circle [id:dp08837630111850925] 
\draw  [fill={rgb, 255:red, 0; green, 0; blue, 0 }  ,fill opacity=1 ] (200.2,92.7) .. controls (200.2,91.65) and (201.05,90.8) .. (202.1,90.8) .. controls (203.15,90.8) and (204,91.65) .. (204,92.7) .. controls (204,93.75) and (203.15,94.6) .. (202.1,94.6) .. controls (201.05,94.6) and (200.2,93.75) .. (200.2,92.7) -- cycle ;
%Straight Lines [id:da6071688854524824] 
\draw    (202.8,139.6) -- (202.3,130.39) ;
%Straight Lines [id:da7621748814947118] 
\draw    (213.6,130.7) -- (202.3,130.39) ;
%Straight Lines [id:da549345234793241] 
\draw    (202.3,130.39) -- (192.9,120.8) ;
%Shape: Circle [id:dp6401169208740147] 
\draw  [fill={rgb, 255:red, 0; green, 0; blue, 0 }  ,fill opacity=1 ] (204.7,130.5) .. controls (204.7,131.55) and (203.85,132.4) .. (202.8,132.4) .. controls (201.75,132.4) and (200.9,131.55) .. (200.9,130.5) .. controls (200.9,129.45) and (201.75,128.6) .. (202.8,128.6) .. controls (203.85,128.6) and (204.7,129.45) .. (204.7,130.5) -- cycle ;
%Straight Lines [id:da9539716072763178] 
\draw    (222.9,102.2) -- (223.4,111.41) ;
%Straight Lines [id:da4176794574685363] 
\draw    (212.1,111.1) -- (223.4,111.41) ;
%Straight Lines [id:da7727263397743658] 
\draw    (223.4,111.41) -- (232.8,121) ;
%Shape: Circle [id:dp8254813535286842] 
\draw  [fill={rgb, 255:red, 0; green, 0; blue, 0 }  ,fill opacity=1 ] (221,111.3) .. controls (221,110.25) and (221.85,109.4) .. (222.9,109.4) .. controls (223.95,109.4) and (224.8,110.25) .. (224.8,111.3) .. controls (224.8,112.35) and (223.95,113.2) .. (222.9,113.2) .. controls (221.85,113.2) and (221,112.35) .. (221,111.3) -- cycle ;
%Straight Lines [id:da8174451327791181] 
\draw    (251.05,130.42) -- (241.85,130.95) ;
%Straight Lines [id:da8905140021489866] 
\draw    (242.12,119.65) -- (241.85,130.95) ;
%Straight Lines [id:da9799995163294477] 
\draw    (241.85,130.95) -- (232.28,140.38) ;
%Shape: Circle [id:dp1689117188293956] 
\draw  [fill={rgb, 255:red, 0; green, 0; blue, 0 }  ,fill opacity=1 ] (241.95,128.55) .. controls (242.99,128.55) and (243.85,129.39) .. (243.85,130.44) .. controls (243.85,131.49) and (243.01,132.35) .. (241.96,132.35) .. controls (240.91,132.35) and (240.05,131.5) .. (240.05,130.46) .. controls (240.05,129.41) and (240.9,128.55) .. (241.95,128.55) -- cycle ;
%Straight Lines [id:da5547434538006091] 
\draw    (253.09,112.74) -- (262.29,112.28) ;
%Straight Lines [id:da2723989340117585] 
\draw    (262.3,120.5) -- (262.29,112.28) ;
%Straight Lines [id:da07799695074159985] 
\draw    (262.29,112.28) -- (271.93,102.92) ;
%Shape: Circle [id:dp4948554229386488] 
\draw  [fill={rgb, 255:red, 0; green, 0; blue, 0 }  ,fill opacity=1 ] (262.18,114.68) .. controls (261.13,114.67) and (260.28,113.82) .. (260.29,112.77) .. controls (260.29,111.72) and (261.14,110.87) .. (262.19,110.88) .. controls (263.24,110.88) and (264.09,111.74) .. (264.09,112.78) .. controls (264.08,113.83) and (263.23,114.68) .. (262.18,114.68) -- cycle ;
%Straight Lines [id:da034782814674966245] 
\draw    (232.69,92.34) -- (241.89,91.88) ;
%Straight Lines [id:da05967801485970825] 
\draw    (241.54,103.18) -- (241.89,91.88) ;
%Straight Lines [id:da15954258268041377] 
\draw    (241.89,91.88) -- (251.53,82.52) ;
%Shape: Circle [id:dp7404841901744301] 
\draw  [fill={rgb, 255:red, 0; green, 0; blue, 0 }  ,fill opacity=1 ] (241.78,94.28) .. controls (240.73,94.27) and (239.88,93.42) .. (239.89,92.37) .. controls (239.89,91.32) and (240.74,90.47) .. (241.79,90.48) .. controls (242.84,90.48) and (243.69,91.34) .. (243.69,92.38) .. controls (243.68,93.43) and (242.83,94.28) .. (241.78,94.28) -- cycle ;
%Straight Lines [id:da6826232948946793] 
\draw    (272.65,149.57) -- (263.45,150.15) ;
%Straight Lines [id:da8958090497245419] 
\draw    (263.67,138.85) -- (263.45,150.15) ;
%Straight Lines [id:da539324135291934] 
\draw    (263.45,150.15) -- (252.2,160.81) ;
%Shape: Circle [id:dp22356645319275437] 
\draw  [fill={rgb, 255:red, 0; green, 0; blue, 0 }  ,fill opacity=1 ] (263.54,147.75) .. controls (264.59,147.74) and (265.45,148.58) .. (265.45,149.63) .. controls (265.46,150.68) and (264.62,151.54) .. (263.57,151.55) .. controls (262.52,151.56) and (261.66,150.71) .. (261.65,149.66) .. controls (261.65,148.61) and (262.49,147.76) .. (263.54,147.75) -- cycle ;
%Straight Lines [id:da30199539586700497] 
\draw    (222.43,160.19) -- (221.98,150.98) ;
%Straight Lines [id:da9947988441548319] 
\draw    (233.28,151.35) -- (221.98,150.98) ;
%Straight Lines [id:da09372540009377706] 
\draw    (221.98,150.98) -- (212,140.4) ;
%Shape: Circle [id:dp973909504724523] 
\draw  [fill={rgb, 255:red, 0; green, 0; blue, 0 }  ,fill opacity=1 ] (224.38,151.1) .. controls (224.37,152.15) and (223.52,153) .. (222.47,152.99) .. controls (221.42,152.99) and (220.57,152.13) .. (220.58,151.08) .. controls (220.59,150.03) and (221.44,149.19) .. (222.49,149.19) .. controls (223.54,149.2) and (224.39,150.05) .. (224.38,151.1) -- cycle ;

% Text Node
\draw (170.4,102.61) node [anchor=north west][inner sep=0.75pt]    {\Large $C_{7}$};

\end{tikzpicture}

        \caption{A construction of $C_8$-free triangle-free plane graphs}
        \label{fig:triC8Construction}
    \end{figure}
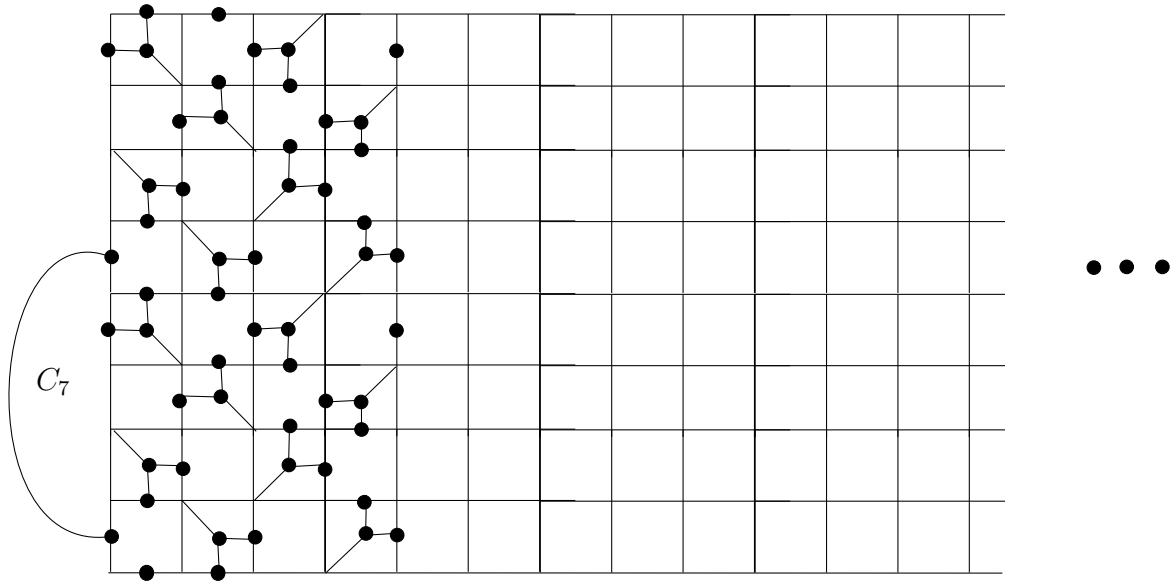
    
    In the construction, the top row of vertices are correspondingly identified with the bottom row of vertices. There are $20t+12$ vertices and $36t+14$ vertices in the construction, where $4t$ is the number of columns for any positive integer $t$. These satisfy that $e_G=9n/5-38/5$.

\subsubsection{Preparatory propositions}

\begin{proposition}
A quadrangular block in a $C_8$-free triangle-free plane graph contains at most $7$ vertices.
\end{proposition}

Proof of this proposition is exactly the same as the proof of Proposition \ref{pro:biC8Block}.

%Assume for contradiction that there exists a triangular block $B'$ in a $C_5$ free plane graph with $n \geq 5$ number of vertices. We delete vertices on the unbounded face of $B'$ one by one until there are exactly $5$ vertices left in $B'$, and we denote the new triangular block by $B$. Without loss of generality, since when we construct the blocks, only $3$-faces are considered, there are only two possibilities for $B$. One is a $K_4$ with a triangle attached, and the other is (figures needed). Then in both cases, we could find a $C_5$, contradiction.

Now we describe all the possible quadrangular blocks in $G$. As shown in Figure \ref{fig:quadraBlocksTriCEight}, other than the trivial quadrangular block $K_2$, there is one $4$-vertex quadrangular block $C_4$, one $6$-vertex quadrangular block $\Theta_6$,  and one $7$-vertex quadrangular block that is denoted by $Q_7$.

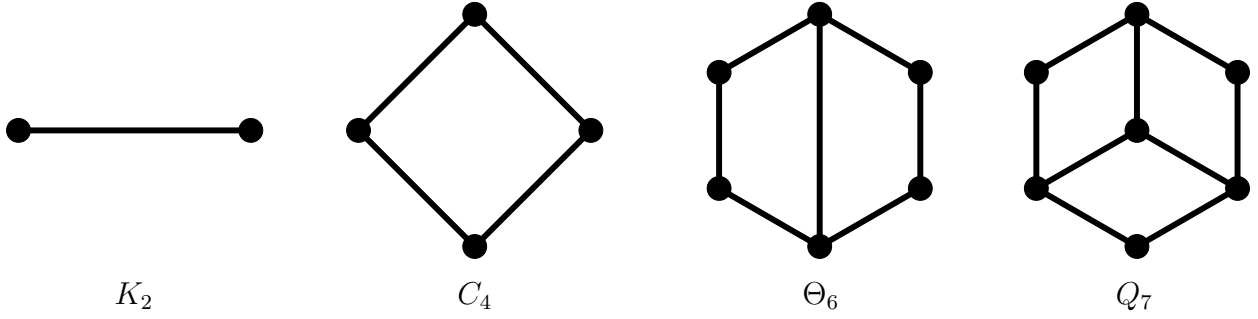
\begin{figure}[H]
    \centering
\begin{tikzpicture}[x=0.75pt,y=0.75pt,yscale=-1,xscale=1]
%uncomment if require: \path (0,300); %set diagram left start at 0, and has height of 300

%Straight Lines [id:da11987314299222929] 
\draw [line width=2.25]    (431.91,84.91) -- (431.91,202.09) ;
%Shape: Circle [id:dp975250785945412] 
\draw  [fill={rgb, 255:red, 0; green, 0; blue, 0 }  ,fill opacity=1 ] (585.91,143.5) .. controls (585.91,140.19) and (588.6,137.5) .. (591.91,137.5) .. controls (595.23,137.5) and (597.91,140.19) .. (597.91,143.5) .. controls (597.91,146.81) and (595.23,149.5) .. (591.91,149.5) .. controls (588.6,149.5) and (585.91,146.81) .. (585.91,143.5) -- cycle ;
%Shape: Circle [id:dp2165776269609594] 
\draw  [fill={rgb, 255:red, 0; green, 0; blue, 0 }  ,fill opacity=1 ] (375.17,114.21) .. controls (375.17,110.89) and (377.86,108.21) .. (381.17,108.21) .. controls (384.49,108.21) and (387.17,110.89) .. (387.17,114.21) .. controls (387.17,117.52) and (384.49,120.21) .. (381.17,120.21) .. controls (377.86,120.21) and (375.17,117.52) .. (375.17,114.21) -- cycle ;
%Shape: Circle [id:dp9843406924738909] 
\draw  [fill={rgb, 255:red, 0; green, 0; blue, 0 }  ,fill opacity=1 ] (476.65,172.79) .. controls (476.65,169.48) and (479.34,166.79) .. (482.65,166.79) .. controls (485.96,166.79) and (488.65,169.48) .. (488.65,172.79) .. controls (488.65,176.11) and (485.96,178.79) .. (482.65,178.79) .. controls (479.34,178.79) and (476.65,176.11) .. (476.65,172.79) -- cycle ;
%Shape: Circle [id:dp8008159406276838] 
\draw  [fill={rgb, 255:red, 0; green, 0; blue, 0 }  ,fill opacity=1 ] (476.65,114.21) .. controls (476.65,110.89) and (479.34,108.21) .. (482.65,108.21) .. controls (485.96,108.21) and (488.65,110.89) .. (488.65,114.21) .. controls (488.65,117.52) and (485.96,120.21) .. (482.65,120.21) .. controls (479.34,120.21) and (476.65,117.52) .. (476.65,114.21) -- cycle ;
%Shape: Circle [id:dp47502628685617565] 
\draw  [fill={rgb, 255:red, 0; green, 0; blue, 0 }  ,fill opacity=1 ] (425.91,84.91) .. controls (425.91,81.6) and (428.6,78.91) .. (431.91,78.91) .. controls (435.23,78.91) and (437.91,81.6) .. (437.91,84.91) .. controls (437.91,88.23) and (435.23,90.91) .. (431.91,90.91) .. controls (428.6,90.91) and (425.91,88.23) .. (425.91,84.91) -- cycle ;
%Straight Lines [id:da11468656816094547] 
\draw [line width=2.25]    (27.82,143.5) -- (145,143.5) ;
%Shape: Circle [id:dp0154309753346622] 
\draw  [fill={rgb, 255:red, 0; green, 0; blue, 0 }  ,fill opacity=1 ] (21.82,143.5) .. controls (21.82,140.19) and (24.51,137.5) .. (27.82,137.5) .. controls (31.14,137.5) and (33.82,140.19) .. (33.82,143.5) .. controls (33.82,146.81) and (31.14,149.5) .. (27.82,149.5) .. controls (24.51,149.5) and (21.82,146.81) .. (21.82,143.5) -- cycle ;
%Shape: Circle [id:dp861593456120753] 
\draw  [fill={rgb, 255:red, 0; green, 0; blue, 0 }  ,fill opacity=1 ] (139,143.5) .. controls (139,140.19) and (141.69,137.5) .. (145,137.5) .. controls (148.31,137.5) and (151,140.19) .. (151,143.5) .. controls (151,146.81) and (148.31,149.5) .. (145,149.5) .. controls (141.69,149.5) and (139,146.81) .. (139,143.5) -- cycle ;
%Shape: Regular Polygon [id:dp470777050189477] 
\draw  [line width=2.25]  (431.91,202.09) -- (381.17,172.79) -- (381.17,114.21) -- (431.91,84.91) -- (482.65,114.21) -- (482.65,172.79) -- cycle ;
%Shape: Circle [id:dp6526240060547581] 
\draw  [fill={rgb, 255:red, 0; green, 0; blue, 0 }  ,fill opacity=1 ] (375.17,172.79) .. controls (375.17,169.48) and (377.86,166.79) .. (381.17,166.79) .. controls (384.49,166.79) and (387.17,169.48) .. (387.17,172.79) .. controls (387.17,176.11) and (384.49,178.79) .. (381.17,178.79) .. controls (377.86,178.79) and (375.17,176.11) .. (375.17,172.79) -- cycle ;
%Shape: Circle [id:dp45329268817311497] 
\draw  [fill={rgb, 255:red, 0; green, 0; blue, 0 }  ,fill opacity=1 ] (425.91,202.09) .. controls (425.91,198.77) and (428.6,196.09) .. (431.91,196.09) .. controls (435.23,196.09) and (437.91,198.77) .. (437.91,202.09) .. controls (437.91,205.4) and (435.23,208.09) .. (431.91,208.09) .. controls (428.6,208.09) and (425.91,205.4) .. (425.91,202.09) -- cycle ;
%Shape: Circle [id:dp014609434012019529] 
\draw  [fill={rgb, 255:red, 0; green, 0; blue, 0 }  ,fill opacity=1 ] (535.17,114.21) .. controls (535.17,110.89) and (537.86,108.21) .. (541.17,108.21) .. controls (544.49,108.21) and (547.17,110.89) .. (547.17,114.21) .. controls (547.17,117.52) and (544.49,120.21) .. (541.17,120.21) .. controls (537.86,120.21) and (535.17,117.52) .. (535.17,114.21) -- cycle ;
%Shape: Circle [id:dp6078909543158129] 
\draw  [fill={rgb, 255:red, 0; green, 0; blue, 0 }  ,fill opacity=1 ] (636.65,172.79) .. controls (636.65,169.48) and (639.34,166.79) .. (642.65,166.79) .. controls (645.96,166.79) and (648.65,169.48) .. (648.65,172.79) .. controls (648.65,176.11) and (645.96,178.79) .. (642.65,178.79) .. controls (639.34,178.79) and (636.65,176.11) .. (636.65,172.79) -- cycle ;
%Shape: Circle [id:dp5983804469232601] 
\draw  [fill={rgb, 255:red, 0; green, 0; blue, 0 }  ,fill opacity=1 ] (636.65,114.21) .. controls (636.65,110.89) and (639.34,108.21) .. (642.65,108.21) .. controls (645.96,108.21) and (648.65,110.89) .. (648.65,114.21) .. controls (648.65,117.52) and (645.96,120.21) .. (642.65,120.21) .. controls (639.34,120.21) and (636.65,117.52) .. (636.65,114.21) -- cycle ;
%Shape: Circle [id:dp9407569106637883] 
\draw  [fill={rgb, 255:red, 0; green, 0; blue, 0 }  ,fill opacity=1 ] (585.91,84.91) .. controls (585.91,81.6) and (588.6,78.91) .. (591.91,78.91) .. controls (595.23,78.91) and (597.91,81.6) .. (597.91,84.91) .. controls (597.91,88.23) and (595.23,90.91) .. (591.91,90.91) .. controls (588.6,90.91) and (585.91,88.23) .. (585.91,84.91) -- cycle ;
%Shape: Regular Polygon [id:dp6163269226629693] 
\draw  [line width=2.25]  (591.91,202.09) -- (541.17,172.79) -- (541.17,114.21) -- (591.91,84.91) -- (642.65,114.21) -- (642.65,172.79) -- cycle ;
%Shape: Circle [id:dp1895153958299296] 
\draw  [fill={rgb, 255:red, 0; green, 0; blue, 0 }  ,fill opacity=1 ] (535.17,172.79) .. controls (535.17,169.48) and (537.86,166.79) .. (541.17,166.79) .. controls (544.49,166.79) and (547.17,169.48) .. (547.17,172.79) .. controls (547.17,176.11) and (544.49,178.79) .. (541.17,178.79) .. controls (537.86,178.79) and (535.17,176.11) .. (535.17,172.79) -- cycle ;
%Shape: Circle [id:dp6326946079418425] 
\draw  [fill={rgb, 255:red, 0; green, 0; blue, 0 }  ,fill opacity=1 ] (585.91,202.09) .. controls (585.91,198.77) and (588.6,196.09) .. (591.91,196.09) .. controls (595.23,196.09) and (597.91,198.77) .. (597.91,202.09) .. controls (597.91,205.4) and (595.23,208.09) .. (591.91,208.09) .. controls (588.6,208.09) and (585.91,205.4) .. (585.91,202.09) -- cycle ;
%Straight Lines [id:da5663609511704109] 
\draw [line width=2.25]    (591.91,84.91) -- (591.91,143.5) ;
%Straight Lines [id:da6055731372971558] 
\draw [line width=2.25]    (591.91,143.5) -- (541.17,172.79) ;
%Straight Lines [id:da3220701634580381] 
\draw [line width=2.25]    (591.91,143.5) -- (642.65,172.79) ;
%Shape: Square [id:dp31398710544055564] 
\draw  [line width=2.25]  (257.91,84.91) -- (316.5,143.5) -- (257.91,202.09) -- (199.32,143.5) -- cycle ;
%Shape: Circle [id:dp5209260697323468] 
\draw  [fill={rgb, 255:red, 0; green, 0; blue, 0 }  ,fill opacity=1 ] (193.32,143.5) .. controls (193.32,140.19) and (196.01,137.5) .. (199.32,137.5) .. controls (202.64,137.5) and (205.32,140.19) .. (205.32,143.5) .. controls (205.32,146.81) and (202.64,149.5) .. (199.32,149.5) .. controls (196.01,149.5) and (193.32,146.81) .. (193.32,143.5) -- cycle ;
%Shape: Circle [id:dp7595450927069378] 
\draw  [fill={rgb, 255:red, 0; green, 0; blue, 0 }  ,fill opacity=1 ] (310.5,143.5) .. controls (310.5,140.19) and (313.19,137.5) .. (316.5,137.5) .. controls (319.81,137.5) and (322.5,140.19) .. (322.5,143.5) .. controls (322.5,146.81) and (319.81,149.5) .. (316.5,149.5) .. controls (313.19,149.5) and (310.5,146.81) .. (310.5,143.5) -- cycle ;
%Shape: Circle [id:dp12027768505551384] 
\draw  [fill={rgb, 255:red, 0; green, 0; blue, 0 }  ,fill opacity=1 ] (251.91,202.09) .. controls (251.91,198.77) and (254.6,196.09) .. (257.91,196.09) .. controls (261.23,196.09) and (263.91,198.77) .. (263.91,202.09) .. controls (263.91,205.4) and (261.23,208.09) .. (257.91,208.09) .. controls (254.6,208.09) and (251.91,205.4) .. (251.91,202.09) -- cycle ;
%Shape: Circle [id:dp2889193861108015] 
\draw  [fill={rgb, 255:red, 0; green, 0; blue, 0 }  ,fill opacity=1 ] (251.91,84.91) .. controls (251.91,81.6) and (254.6,78.91) .. (257.91,78.91) .. controls (261.23,78.91) and (263.91,81.6) .. (263.91,84.91) .. controls (263.91,88.23) and (261.23,90.91) .. (257.91,90.91) .. controls (254.6,90.91) and (251.91,88.23) .. (251.91,84.91) -- cycle ;

% Text Node
\draw (75,219) node [anchor=north west][inner sep=0.75pt]  [font=\Large] [align=left] {$K_2$};
% Text Node
\draw (248,219) node [anchor=north west][inner sep=0.75pt]  [font=\Large] [align=left] {$C_4$};
% Text Node
\draw (422,219) node [anchor=north west][inner sep=0.75pt]  [font=\Large] [align=left] {$\Theta_6$};
% Text Node
\draw (580,219) node [anchor=north west][inner sep=0.75pt]  [font=\Large] [align=left] {$Q_7$};
\end{tikzpicture}
    \caption{All the possible quadrangular blocks in $G$.}
    \label{fig:quadraBlocksTriCEight}
\end{figure}

The following proposition is useful in the proof.
\begin{proposition}
\label{pro:triC8}
Let $G$ be a $C_8$-free triangle-free plane graph on $n$ vertices with $\delta(G)\geq 3$.
\begin{enumerate}
    \item If $B$ is a nontrivial quadrangular block, then none of the exterior faces can have length $6$.
    \item If $B$ is $\Theta_6$ or $Q_7$, then at most one of its exterior faces can be a $5$-face sharing exactly two consecutive edges with $B$.
    \item No two $5$-faces can be adjacent.
\end{enumerate}
\end{proposition}

\begin{proof}~
\begin{enumerate}
    \item Suppose one exterior face $f$ of a nontrivial quadrangular block $B$ is of length $6$, then either together with its adjacent interior $4$-face of $B$, we get a cycle of length $8$, or we get a degree $2$ vertex, or else we have a triangle in the graph, contradiction.
    \item If there are two exterior $5$-faces $f_1,f_2$ of $B$, each sharing exactly two consecutive edges with $B$, then the boundary of $f_1\cup B\cup f_2$ forms a $C_8$.
    \item If two $5$-faces share exactly one edge, they would form a $C_8$, contradiction. If they share more than one common edge, then they will create a multiple edge, or a triangle, or a degree $2$ vertex, hence contradiction.
\end{enumerate}
\end{proof}

\subsubsection{Proof of Theorem \ref{thm:triC8}}~

Our main goal is to show that
\begin{lemma}
Let $G$ be a $C_8$-free triangle-free plane graph on $n$ vertices with $\delta(G)\geq 3$. Any quadrangular block $B$ in $G$ satisfies $$
24v(B)-61e(B)+105f(B)\le 0.
$$.
\end{lemma}
Once we have proved this, we have $$
24v_G-61e_G+105f_G=\sum_{B\in \mathcal{Q}}(24v(B)-61e(B)+105f(B))\le 0.
$$ Combining this with Euler's formula $v_G-e_G+f_G=2$ finishes the proof.

\begin{proof}
We do casework to proceed with the proof. We distinguish the cases according to $B$.

\textbf{Case 1}: $B$ is $K_2$.

Since $\delta(G)\ge 3$, we know that each vertex of $B$ is contained in at least two quadrangular blocks, which means both of them have to be junction vertices, so we have $v(B)\le 1$. By Proposition \ref{pro:triC8}, we know that no two $5$-faces can be adjacent, so one of the faces containing the only edge of $B$ has length at least $5$ and the other has length at least $6$. Suppose both exterior faces of $B$ do not have length $6$, then we know that $f(B)\le 1/5+1/7$. With $e(B)=1$, we have $24v(B)-61e(B)+105f(B)\le -1$. 

If there is an exterior face $f$ of $B$ having length $6$, then $f(B)\le 1/5+1/6$. By Proposition \ref{pro:triC8}, each block having $f$ as an exterior face must be $K_2$. Thus, the two junction vertices of $B$ are both shared by at least three quadrangular blocks, which yields $v(B)\le 2/3$. With $e(B)=1$, we have $24v(B)-61e(B)+105f(B)\le -13/2$.

\textbf{Case 2}: $B$ is $C_4$.

Since all of the four vertices in $B$ have degree $2$, we know all of them are junction vertices. So we have $v(B)\le 4/2=2$. There is only one interior face in $B$ and each exterior face of $B$ has length at least $5$ so we have $f(B)\le 1+4/5$. With $e(B)=4$, we have $24v(B)-61e(B)+105f(B)\le-7$.

\textbf{Case 3}: $B$ is $\Theta_6$.

Four of the vertices in $B$ have degree $2$, so we know they must be junction vertices, while each of the other two degree $3$ vertices can either be a junction vertex or not. We consider $2$ different scenarios.
\begin{enumerate}[(i)]
    \item There are exactly $4$ junction vertices in $B$. In this case, we have $v(B)\le 2+4/2$. Note that now there are two pairs of consecutive exterior edges. If they lie in the same exterior face, then it's easy to check this face has length at least $7$. If they lie in different exterior faces, then by Proposition \ref{pro:triC8} we know that at most one of the two faces can have length $5$. So, at least one of those two faces have length at least $7$. We then have $f(B)\le 2+2/7+4/5$. With $e(B)=7$, we have $24v(B)-61e(B)+105f(B)\le-7$.
    \item There are $5$ or $6$ junction vertices. We have $v(B)\le 1+5/2$ and $f(B)\le 2+6/5$. With $e(B)=7$ we have $24v(B)-61e(B)+105f(B)\le-7$.
\end{enumerate}
%potentialFigs

\textbf{Case 4}: $B$ is $Q_7$.

We always have $e(B)=9$ in this case. Three of the exterior vertices in $B$ have degree $2$, so we know they must be junction vertices. Now we consider $3$ different scenarios.
\begin{enumerate}[(i)]
    \item There are exactly $3$ junction vertices in $B$. In this case, we have $v(B)\le 4+3/2$. By Proposition \ref{pro:triC8} we know that at most one of the three exterior faces containing two consecutive exterior edges has length $5$. Thus, we have $f(B)\le 3+2/5+4/7$. We then can see $24v(B)-61e(B)+105f(B)\le 0$.
    \item There are $4$ junction vertices in $B$. In this case, we have $v(B)\le 3+4/2$. We know that at most one of the two exterior faces containing a pair of consecutive exterior edges has length $5$, so we have $f(B)\le 3+4/5+2/7$, which gives $24v(B)-61e(B)+105f(B)\le 0$.
    \item There are $5$ or $6$ junction vertices in $B$. In this case, we have $v(B)\le 2+5/2$ and $f(B)\le 3+6/5$. Those give us $24v(B)-61e(B)+105f(B)\le 0$.
\end{enumerate}
%potentialFigs
\end{proof}

\begin{remark}
It's easy to see that there is no graph only containing $Q_7$ quadrangular blocks, so the upper bound we have now is not sharp. It's still open how to improve it.
\end{remark}

\section*{Acknowledgements}

This paper was written under the auspices of the Budapest Semesters in Mathematics program during the Spring of 2022. The research of Gy\H{o}ri  was partially supported by the National Research, Development and Innovation Office NKFIH, grants  K132696, SNN135643.

\printbibliography

\end{document}